\documentclass[12pt]{amsart}
\usepackage[utf8]{inputenc}

\usepackage[margin=1.3in]{geometry}

\usepackage{amsmath}
\usepackage{amsfonts}
\usepackage{amssymb}
\usepackage{amsthm}
\usepackage{mathtools}
\usepackage{caption}
\usepackage{subcaption}
\usepackage{bbm}
\usepackage[export]{adjustbox}

\usepackage{stmaryrd}

\usepackage[all]{xy}

\usepackage{tikz-cd}
\usetikzlibrary{matrix}
\usepackage{graphicx} 
\usepackage{epstopdf}

\usepackage[linktocpage]{hyperref}
\hypersetup{
    colorlinks=true,
    linkcolor=blue,
    citecolor=blue,      
    urlcolor=blue,
}

\usepackage{color}
\definecolor{note}{rgb}{0,0,1}  

\newtheorem{theorem}{Theorem}
\newtheorem{definition}[theorem]{Definition}
\newtheorem{proposition}[theorem]{Proposition}
\newtheorem{lemma}[theorem]{Lemma}
\newtheorem{claim}[theorem]{Claim}
\newtheorem{corollary}[theorem]{Corollary}

\newtheorem{remark}[theorem]{Remark}

\numberwithin{equation}{section}
\numberwithin{theorem}{section}

\usepackage{enumitem}
\newcommand{\mylabel}[2]{#2\def\@currentlabel{#2}\label{#1}}

\usepackage{todonotes}

\usepackage[english]{babel}

\usepackage{hyphenat}

\usepackage[backend=bibtex,style=alphabetic,maxalphanames=4,maxnames=4]{biblatex}

\renewbibmacro{in:}{}

\DeclareDelimFormat[bib,biblist]{nametitledelim}{\addcomma\space}

\DeclareFieldFormat*{title}{\mkbibitalic{#1}\addcomma}
\DeclareFieldFormat*{journaltitle}{#1}
\DeclareFieldFormat*{volume}{\mkbibbold{#1}}
\DeclareFieldFormat{pages}{#1}
\DeclareFieldFormat[misc]{date}{preprint {#1}}
\DeclareFieldFormat{mr}{%
  MR\addcolon\space
  \ifhyperref
    {\href{http://www.ams.org/mathscinet-getitem?mr=MR#1}{\nolinkurl{#1}}}
    {\nolinkurl{#1}}}
    
\AtEveryBibitem{
  \clearfield{url}
  \clearfield{number}
  \clearfield{doi}
  \clearfield{issn}
  \clearfield{isbn}
  \clearfield{eprintclass}
}
\AtEveryBibitem{\ifentrytype{book}{\clearfield{pages}}{}}

\bibliography{hecke}

\usepackage{fancyhdr}
\usepackage{todonotes}

\title{Folded Morse flow trees}

\author{Tianyu Yuan}
\address{Beijing International Center for Mathematical Research, Peking University, Beijing 100871, China}
\email{ytymath@pku.edu.cn} \urladdr{}

\date{\today}

\subjclass[2010]{Primary 58E05; Secondary 53D40.}

\begin{document}

\maketitle

\begin{abstract}
We present an approach to Morse theory on symmetric products of surfaces using the notion of folded ribbon trees.
We introduce an $A_\infty$-category with objects defined as $\kappa$-tuples of Morse functions, where the differential of the tuple has no self-intersection.
We show that when the graph of the differential of the $\kappa$-tuple of Morse functions on $T^*\mathbb{R}^2$ is the wrapped $\kappa$ disjoint cotangent fibers, its endormorphism is the Hecke algebra associated to the symmetric group $\mathfrak{S}_\kappa$.
\end{abstract}

\tableofcontents

\section{Introduction}

Given a Riemannian manifold $M$, Fukaya \cite{fukaya1996morse} constructed an $A_\infty$-category where the objects are Morse functions on $M$, the morphisms are generated by critical points of differences between Morse functions, and the composition maps are defined by counting Morse gradient flows associated to ribbon trees. 
To generalize from $M$ to the $\kappa$-th symmetric product $\operatorname{Sym}^\kappa(M)$, we introduce the notion of {\em folded ribbon trees}, which contains information of how the corresponding gradient flows intersect the big diagonal of $\operatorname{Sym}^\kappa(M)$.

For technical reasons explained in Appendix \ref{section-appendix-cpt}, we focus on the case when $M$ is of dimension two.
Given a surface $S$ possibly with punctures, we define an $A_\infty$-category $\operatorname{Morse}_\kappa(S)$, where the objects are $\kappa$-tuples of Morse functions $\boldsymbol{f}=(f_1,\dots,f_\kappa)$ on $S$ whose differential has no self-intersection along with some additional assumptions.
The $A_\infty$ structure is then defined by counting {\em folded Morse flow trees}, which are Morse flow graphs on $S$ associated to folded ribbon trees. 
In the first part of the paper we show that:
\begin{theorem}
    $\operatorname{Morse}_\kappa(S)$ defines an $A_\infty$-category.
\end{theorem}


When $S=\mathbb{R}^2$, we compare $\operatorname{Morse}_\kappa(\mathbb{R}^2)$ with the higher-dimensional Heegaard Floer (HDHF) category $\mathcal{F}_\kappa(T^*\mathbb{R}^2)$ due to Colin, Honda, and Tian \cite{colin2020applications}.
An object $\boldsymbol{f}=(f_1,\dots,f_\kappa)$ of $\operatorname{Morse}_\kappa(\mathbb{R}^2)$ naturally corresponds to the object $\Gamma_{d\boldsymbol{f}}=(\Gamma_{df_1},\dots,\Gamma_{df_\kappa})$ of $\mathcal{F}^{gr}_\kappa(T^*\mathbb{R}^2)\subset \mathcal{F}_\kappa(T^*\mathbb{R}^2)$, where $\Gamma_{df_i}$ is the graph of $df_i$ in $T^*\mathbb{R}^2$ and $\mathcal{F}^{gr}_\kappa(T^*\mathbb{R}^2)$ is the subcategory of $\mathcal{F}_\kappa(T^*\mathbb{R}^2)$ containing graphical Lagrangians.
With some additional assumptions, we define an $A_\infty$-functor $\mathcal{E}$ from $\mathcal{F}^{gr}_\kappa(T^*\mathbb{R}^2)$ to $\operatorname{Morse}_\kappa(\mathbb{R}^2)$ which on the object level maps $\Gamma_{d\boldsymbol{f}}$ to $\boldsymbol{f}$. 
The higher composition maps are defined by counting a mixed moduli space of pseudoholomorphic curves and Morse flow graphs.
We show that $\mathcal{E}$ is an $A_\infty$-equivalence.

Consider a $\kappa$-tuple of disjoint cotangent fibers, denoted by the disjoint union $\sqcup_{i=1}^\kappa T^*_{q_i}\mathbb{R}^2$. 
The partially wrapping of $\sqcup_{i=1}^\kappa T^*_{q_i}\mathbb{R}^2$ in $T^*\mathbb{R}^2$ can be viewed as an object of $\mathrm{Morse}_\kappa(\mathbb{R}^2)$. 
Then we get an $A_\infty$-algebra $\operatorname{End}_{Mor}(\sqcup_{i=1}^\kappa T^*_{q_i}\mathbb{R}^2)$.
On the other hand, $\sqcup_{i=1}^\kappa T^*_{q_i}\mathbb{R}^2$ is an object of the (wrapped) HDHF category $\mathcal{F}^{gr}_\kappa(T^*\mathbb{R}^2)$.
The endormorphism algebra $\operatorname{End}_{Fuk}(\sqcup_{i=1}^\kappa T^*_{q_i}\mathbb{R}^2)$ turns out to be the finite Hecke algebra $H_\kappa$ due to \cite{honda2022jems} and \cite{tian2022example}.
Therefore, we can state the main result:

\begin{theorem}
    $\operatorname{End}_{Mor}(\sqcup_{i=1}^\kappa T^*_{q_i}\mathbb{R}^2)\simeq H_\kappa$.
\end{theorem}

The main purpose of this paper is to introduce the notion of folded Morse flow trees and we only consider the simplest case where objects are tuples of Morse functions. 
The examples are quite restrictive since we require that the differential of the Morse functions in a tuple has no self-intersection. 
In general, we can consider a broader range of candidates for objects, e.g., (1)
the front projection of Legendrian submanifolds in 1-jet spaces; (2) the multi-valued functions corresponding to wrapped cotangent fibers in cotangent bundles.

\vskip.15in
We now briefly discuss another application of folded Morse flow trees, which will appear in a future work. 
Consider a spectral curve, which is an $N$-fold holomorphic branched cover $\pi\colon\Sigma\to C$ between complex curves.
In the work of Gaiotto, Moore, and Neitzke \cite{gaiotto2013spectral}, they introduced the notion of spectral networks, which are essentially certain Morse flows, and showed that a $\mathfrak{gl}(1)$-local system on $\Sigma$ induces a $\mathfrak{gl}(N)$-local system on $C$.
Based on the work of Gaiotto-Moore-Neitzke, Neitzke and Yan \cite{neitzke2020q} then defined a $q$-deformed verion, i.e., there is a map from the $\mathfrak{gl}(N)$-skein algebra on $C\times\mathbb{R}$ to the $\mathfrak{gl}(1)$-skein algebra on $\Sigma\times\mathbb{R}$.
We give another construction similar to Neitzke-Yan. 
Using folded Morse flow trees, we can define a homomorphism from the braid skein algebra of $C$ to the braid skein algebra of $\Sigma$.
The advantage of our definition is that the involved spectral networks are always rigid regardless of the number of loops. 

Furthermore, in the future work, we will give the interpretation of both Neitzke-Yan's and our definition by holomorphic curves.
This depends on a direct comparison between moduli spaces of holomorphic curves and Morse gradient flows, which generalizes the work of Fukaya and Oh \cite{fukaya1997zero}.
We will show that Neitzke-Yan's approach corresponds to open Gromov-Witten theory, based on the work of Ekholm and Shende \cite{ekholm2021skeins}, and our approach by folded Morse flow trees corresponds to higher-dimensional Heegaard Floer theory.


~\\
\noindent \textit{Acknowledgements}. I would like to thank Ko Honda and Yin Tian for numerous ideas and suggestions.
I also thank Siyang Liu, Yoon Jae Nho, Vivek Shende, Tadayuki Watanabe, and Peng Zhou for valuable discussions.
I am partially supported by China Postdoctoral Science Foundation 2023T160002.

\section{Morse theory of multiple functions}
\label{section-morse}

\subsection{Folded ribbon trees}

Our Morse theory involves counting certain Morse flow graphs with a projection to some decorated ribbon trees, which represents the $A_\infty$ structure.

We first recall the definition of a ribbon tree $T$, which is topologically a tree embedded in $D^2$ so that $\partial T\subset\partial D^2$ along with additional requirements. 
Denote the set of vertices and edges of $T$ by $V$ and $E$.
All edges of $E$ are directed.
We write
\begin{equation*}
    V=V_{ext}\sqcup V_{inn},\quad E=E_{ext}\sqcup E_{inn},
\end{equation*}
so that $V_{ext}$ consists of vertices on $\partial D^2$ and $E_{ext}$ consists of edges adjacent to $V_{ext}$. 
We require that each $v\in V_{inn}$ has degree $|v|\geq3$. 
Denote the starting and ending vertex of $e\in E$ by $v_{st}(e)$ and $v_{en}(e)$, respectively.

For each $v\in V_{inn}$, there is a (counter-clockwise) cyclic order of the edges attached to $v$:
\begin{equation*}
    E_v=\{e_{v,0},e_{v,1},\dots,e_{v,|v|-1}\},
\end{equation*}
where $E_v$ is the set of edges adjacent to $v$, so that $e_{v,0}$ is the unique outgoing edge. 
Therefore, there is a unique edge $e_0\in E_{ext}$ which is outgoing (pointing outwards $\partial D^2$) and all other edges of $E_{ext}$ are incoming.
We denote the outgoing vertex of $e_0$ by $v_0$.

Given $v_1,v_2\in V$, denote by $E_{v_1\to v_2}$ the set of edges on the unique path starting from $v_1$ and ending on $v_2$.
Note that $E_{v_1\to v_2}$ may be empty.

The tree $T$ naturally induces partial orders on both $V$ and $E$: for $v_1,v_2\in V$, we write $v_1\preceq v_2$ if $v_1$ is on the path from $v_2$ to $v_0$; for $e_1,e_2\in E$, $e_1\preceq e_2$ if $e_1$ is on the path from $e_2$ to $e_0$.
For $e\in E$, denote $E_{\succeq e}=\{e'\in E\,|\,e\preceq e'\}$.

\begin{figure}
    \centering
    \includegraphics[width=6cm]{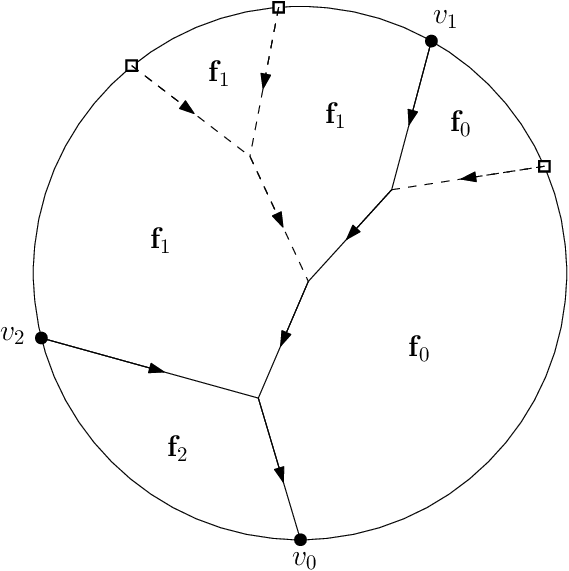}
    \caption{A folded ribbon tree with $|V_{ste}|=3$. The marginal vertices and marginal edges are denoted by boxes and dashed lines.}
    \label{fig-ribbon-tree}
\end{figure}

\begin{definition}
    A folded ribbon tree is a tuple $T_*=(T,\sigma,l)$ which satisfies
    \begin{enumerate}
        \item $T$ is a ribbon tree.
        
        \item $V_{ext}=V_{ste}\sqcup V_{mar}$. 
        We call $V_{ste}$ stem vertices and $V_{mar}$ marginal vertices. 
        We require $v_0\in V_{ste}$. 
        If $e\in E$ is on the path starting from some $v\in V_{ste}$, we say $e$ is a stem edge.
        Denote the set of stem edges by $E_{ste}$ and let $E=E_{ste}\sqcup E_{mar}$.
        Suppose $|V_{ste}|=m+1$, then we label $V_{ste}$ by $v_0,v_1,\dots,v_m$ in the (counter-clockwise) cyclic order of $\partial D^2$. 
        The corresponding adjacent edges are labelled by $e_0,e_1,\dots,e_m$.
        
        \item $\sigma\colon E\to \mathfrak{S}_\kappa$ assigns each edge a permutation so that for each $v\in V_{inn}$, 
        \begin{equation*}
            \sigma(e_{v,0})=\sigma(e_{v,|v|-1})\cdots\sigma(e_{v,1}).
        \end{equation*}
        For all $e\in E_{mar}\cap E_{ext}$, $\sigma(e)$ is a transposition.
        For all $e\in E_{mar}$, $\sigma(e)\neq id$.

        \item Incoming edges in $E_{ext}\cap E_{ste}$ are identified with $[0,\infty)$ and the outgoing edge in $E_{ext}$ is identified with $(-\infty,0]$.
        Each $e\in E_{inn}\cup E_{mar}$ is identified with $[0,l(e)]$, where
        \begin{equation*}
            l\colon E_{inn}\cup E_{mar}\to[0,\infty)
        \end{equation*}
        is a length function.
    \end{enumerate}
    
    See Figure \ref{fig-ribbon-tree} for an example.
\end{definition}

\begin{remark}
    Note that the identification of edges with intervals of $\mathbb{R}$ in condition (4) reverses the orientation.
    From this point onwards, by orientation we always refer to the orientation of edges unless stated otherwise.
\end{remark}

Given a folded ribbon tree $(T,\sigma,l)$ with $|V_{ste}|=m+1$, we define
\begin{equation*}
    i_l,i_r\colon E\to \{0,1,\dots,m\}
\end{equation*}
as follows: For $e\in E$, let $C_l$ ($C_r$) be the component of $D^2\backslash E_{ste}$ whose boundary contains $e$ and has the same (opposite) orientation. 
If the endpoint of the oriented arc $\partial C_l\cap\partial D_2$ ($\partial C_r\cap\partial D_2$) is $v_i$, then we set $i_l(e)=i$ ($i_r(e)=i$).


A folded ribbon tree $T_*$ uniquely determines a ribbon graph $G_*$, which is approximately a $\kappa$-to-1 cover of $T$.

\begin{figure}
    \centering
    \includegraphics[width=12cm]{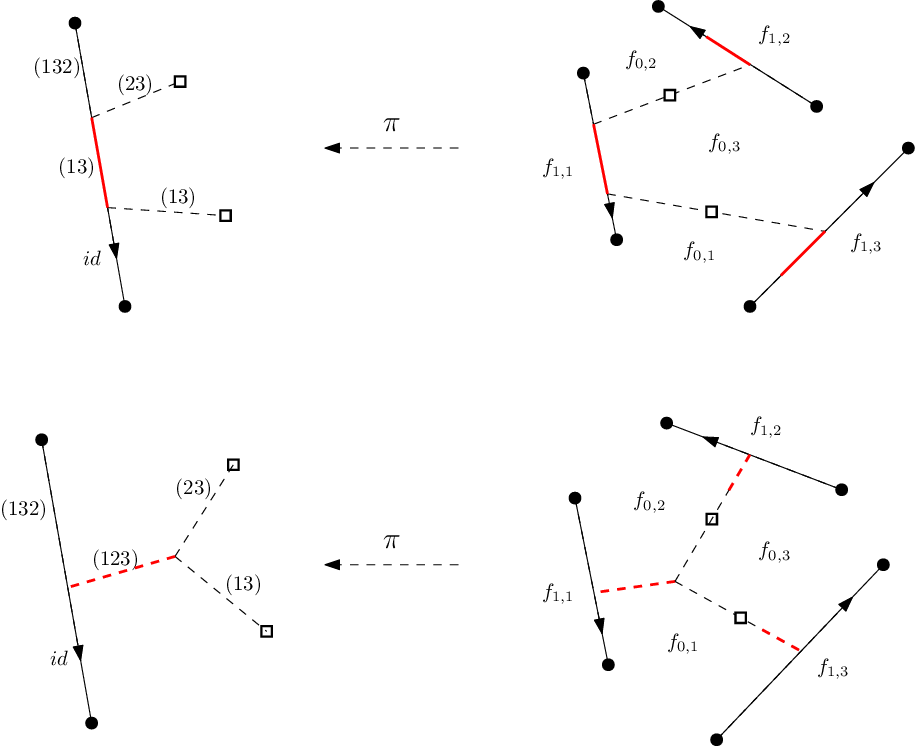}
    \caption{Two ribbon graphs are depicted on the right side, along with their corresponding projections onto folded ribbon trees shown on the left side. Here the notation $f_{i,h}$ determines $i_l,h_l$ (or $i_r,h_r$) for each edge of $G$ by viewing $G$ as embedded in $\mathbb{R}^2$.}
    \label{fig-ribbon-graph}
\end{figure}

\begin{definition}
    A ribbon graph associated to a folded ribbon tree $T_*$ is a tuple $G_*=(G,T_*,\pi,h_l,h_r,i_l,i_r)$ which satisfies
    \begin{enumerate}
        \item G is a graph with vertices $\tilde{V}$ and directed edges $\tilde{E}$.
        
        \item $\pi\colon G\to T$ is a continuous map so that the restriction $\pi|_{\tilde{e}}$ is a diffeomorphism to some $e\in E$ for each $\tilde{e}\in \tilde{E}$. 
        When there is no ambiguity, we also denote $\pi\colon \tilde{E}\to E$ the map between sets of edges induced by $\pi$.
        
        \item $\pi(\tilde{V})\subset V$.

        \item For each $e\in E$, $|\pi^{-1}(e)|=\kappa$. 
        $h_l,\,h_r\colon \pi^{-1}(e)\to \{1,\dots,\kappa\}$ are bijections so that $h_r(\tilde{e})=\sigma(e)(h_l(\tilde{e}))$ for each $\tilde{e}\in \pi^{-1}(e)$.

        \item For each $e\in E_{mar}\cap E_{ext}$, since $\sigma(e)$ is a transposition, there exist two edges $\tilde{e}_1,\tilde{e}_2\in\pi^{-1}(e)$ so that $h_l(\tilde{e}_i)\neq h_r(\tilde{e}_i)$, $i=1,2$.
        We require that their starting points on $\tilde{V}_{mar}$ coincide, i.e., $v_{st}(\tilde{e}_1)=v_{st}(\tilde{e}_2)$.



        \item Each $\tilde{e}\in \tilde{E}$ is identified with $(-\infty,0]$, $[0,\infty)$, or $[0,l(\pi(\tilde{e}))]$ induced from $\pi(\tilde{e})$.

        \item For $\tilde{e}\in \tilde{E}$, we set $i_l(\tilde{e})=i_l(\pi(\tilde{e}))$ and $i_r(\tilde{e})=i_r(\pi(\tilde{e}))$.
    \end{enumerate}
\end{definition}

Note that the Euler characteristic of $G$ is $\chi(G)=\kappa-|V_{mar}|$.

Figure \ref{fig-ribbon-graph} shows an example of a ribbon graph.

\subsection{Perturbation data}
\label{section-morse-data}

When $\kappa=1$, the folded ribbon trees reduce to the case discussed in \cite{fukaya1997zero}, where transversality can be achieved by perturbing Morse functions.
However, for $\kappa\geq1$, perturbing Morse functions is not sufficient to ensure transversality due to the existence of marginal edges.
Instead, we choose a family of domain-dependent metrics to make enough perturbation. 
This is analogous to the choice of Floer data in the construction of the Fukaya category \cite{seidel2008fukaya}.

Denote the set of folded ribbon trees by $\mathcal{T}$. 
Let $p_0$ be the forgetful map
\begin{gather}
    p_0\colon \mathcal{T}\to\mathcal{T}_0,\quad (T,\sigma,l)\mapsto(T,l),
\end{gather}
where $\mathcal{T}_0$ represents the set of all pairs $(T,l)$.
There is a natural topology on both $\mathcal{T}$ and $\mathcal{T}_0$, induced by the length of inner edges. Therefore $p_0$ is a continuous map.
Let $\mathcal{T}^{m,\chi}$ (resp. $\mathcal{T}^{m,\chi}_0$) be the subset of $\mathcal{T}$ (resp. $\mathcal{T}_0$) where $|V_{ste}|=m+1$ and the associated ribbon graph has Euler characteristic $\chi$. 
Note that $\mathcal{T}^{m,\chi}_0$ is well-defined since the Euler characteristic of the associated ribbon graph depends only on $T$.

Consider a surface $S$ with metric $g_0$.
Denote by $\mathfrak{M}$ the set of metrics on $S$ which coincide with $g_0$ outside a compact subset of $S$.

\begin{definition}
\label{definition-morse-datum} 
    A {\bf perturbation datum} on $(T,l)\in \mathcal{T}^{m,\chi}_0$ is a map
    \begin{equation*}
        g^{m,\chi}\colon T\to\mathfrak{M}
    \end{equation*}
    such that $g^{m,\chi}=g_0$ on $E_{ext}\cap E_{ste}$.
    
    A {\bf choice of perturbation data} on $\mathcal{T}^{m,\chi}_0$ is a choice of
    \begin{equation*}
        g^{m,\chi}\colon T\to\mathfrak{M}
    \end{equation*}
    that depends continuously on $(T,l)$ for each $m,\chi$, and satisfies {\em compatibility with boundary strata}: If, as $i\to\infty$, the sequence $\{(T,l^i)\}\subset\mathcal{T}^{m,\chi}$ converges to 
    $$((T_1,l_1),(T_2,l_2))\in\mathcal{T}^{m_1,\chi_1}\times\mathcal{T}^{m_2,\chi_2}$$ with $m_1+m_2=m+1$ and $\chi_1+\chi_2=\chi+\kappa$, i.e., $l^i(e)\to\infty$ for some $e\in E_{inn}\cap E_{ste}$ so that $T_1$ is obtained from $T$ by removing $e$ and attaching an edge $e_-=(-\infty,0]$ to $v_{st}(e)$; $T_2$ is obtained from $T$ by removing $e$ and attaching an edge $e_+=[0,+\infty)$ to $v_{en}(e)$.
    Then we require that
    \begin{enumerate}
        \item On the breaking edge $e$, for each $s$,
        \begin{equation*}
            g^{m,\chi}_{T,l^i}(s)\to g^{m_1,\chi_1}_{T,l_1}(s)=g_0,\quad g^{m,\chi}_{T,l^i}(s-l^i(e))\to g^{m_2,\chi_2}_{T,l_2}(s)=g_0
        \end{equation*}
        as $i\to\infty$, where $g^{m_1,\chi_1}_{T,l_1}$ (resp. $g^{m_2,\chi_2}_{T,l_2}$) is defined on $e_-$ (resp. $e_+$).
        
        \item If $e'\in E\backslash\{e\}$ limits to $T_1$ (resp. $T_2$), for each $s$,
        \begin{equation*}
            g^{m,\chi}_{T,l^i}(s)\to g^{m_1,\chi_1}_{T,l_1}(s),\quad \left(\text{resp. }g^{m,\chi}_{T,l^i}(s)\to g^{m_2,\chi_2}_{T,l_2}(s)\right)
        \end{equation*}
        as $i\to\infty$.
    \end{enumerate}
    The definition of compatibility is similar if $\{(T,l^i)\}\subset\mathcal{T}^{m,\chi}$ breaks into more than two terms.
\end{definition}

\begin{lemma}
\label{lemma-perturbation-data}
    Suppose ${g}^{m,\chi}$ has been chosen for $m\leq k$, $\chi\geq r$, and satisfies the above conditions, then it can be extended to a choice of perturbation data for all $m,\chi$.
\end{lemma}
\begin{proof}
    The construction is similar to but easier than that of choosing Floer data on a family of Riemann surfaces.
    If ${g}^{m,\chi}$ has been chosen for $m\leq k$, $\chi\geq r$, and now we want to choose ${g}^{m,\chi}$ for some $m\leq k+1,\chi\geq r-1$, all we need is that ${g}^{m,\chi}$ is compatible with the boundary strata $\mathcal{T}^{m_1,\chi_1}\times\mathcal{T}^{m_2,\chi_2}$, which has already been chosen by assumption.
    This can be easily achieved since $\mathfrak{M}$ is a convex set.
\end{proof}

When there is no ambiguity, we write $g^{m,\chi}$ as $g$ and also denote
\begin{equation*}
    g\colon G\to\mathfrak{M},\quad x\mapsto g(\pi(x)).
\end{equation*}

\subsection{The $A_\infty$ structure}

Given $(S,g_0)$ and fixing a choice of perturbation data, we define the $A_\infty$-category $\operatorname{Morse}_\kappa(S)$ as follows:

The objects of $\operatorname{Morse}_\kappa(S)$ are $\kappa$-tuples of Morse functions $\boldsymbol{f}=(f_1,\dots,f_\kappa)$ on $S$ such that $df_i-df_j$ is nonvanishing for each pair $i\neq j$.
For two objects $\boldsymbol{f}_0,\boldsymbol{f}_1$, the morphism space $\operatorname{hom}(\boldsymbol{f}_0,\boldsymbol{f}_1)$ is the vector space over $\mathbb{Z}_2\llbracket\hbar\rrbracket$ generated by $\kappa$-tuples of intersection points between $d\boldsymbol{f}_0$ and $d\boldsymbol{f}_1$. 
In other words, if $\boldsymbol{q}\in\operatorname{hom}(\boldsymbol{f}_0,\boldsymbol{f}_1)$, then $\boldsymbol{q}=(q_1,\dots,q_\kappa)$ where $q_i\in \operatorname{Crit}(f_{0,i},f_{1,\sigma(i)})$ for some $\sigma\in \mathfrak{S}_\kappa$ and $i=1,\dots,\kappa$.

Consider $m+1$ objects $\{\boldsymbol{f}_0,\dots,\boldsymbol{f}_m\}$ of $\operatorname{Morse}_\kappa(S)$.
Denote $\boldsymbol{f}_{m+1}\coloneqq\boldsymbol{f}_0$. 
We assume that $f_{i,j}-f_{i+1,k}$ is Morse-Smale with respect to $g_0$ for all $i,j,k$, i.e.,
\begin{itemize}
    \item The critical points of $f_{i,j}-f_{i+1,k}$ are Morse;
    \item The unstable manifolds and stable manifolds of $-\nabla_{g_0}(f_{i,j}-f_{i+1,k})$ intersect transversely.
\end{itemize}

Let $\boldsymbol{q}_i\in \mathrm{hom}(\boldsymbol{f}_{i-1},\boldsymbol{f}_{i})$ for $i=1,\dots,m$ and $\boldsymbol{q}_0\in \mathrm{hom}(\boldsymbol{f}_0,\boldsymbol{f}_m)$.
The moduli space $\mathcal{M}(\boldsymbol{q}_1,\dots,\boldsymbol{q}_m;\boldsymbol{q}_0)$ consists of tuples $(G_*,\gamma)$ that satisfy the following conditions:

\begin{enumerate}
    \item $G_*$ is a ribbon graph with $|V_{ste}|=m+1$.

    \item $\gamma\colon G\to S$ is a continuous map such that
    \begin{equation*}
        \dot{\gamma}|_{\tilde{e}}=\nabla_g(f_{i_r(\tilde{e}),h_r(\tilde{e})}-f_{i_l(\tilde{e}),h_l(\tilde{e})}),
    \end{equation*}
    on each $\tilde{e}\in \tilde{E}$ identified with an interval of $\mathbb{R}$ specified by the underlying folded ribbon tree $T_*$.

    \item For $\pi(\tilde{v})\in V_{ste}$, denote the unique edge adjacent to $\tilde{v}$ by $\tilde{e}\in\tilde{E}$. 
    If $\tilde{v}$ is incoming then $\gamma(\tilde{v})=q_{i_r(\tilde{e}),h_l(\tilde{e})}$. If $\tilde{v}$ is outgoing then $\gamma(\tilde{v})=q_{0,h_l(\tilde{e})}$.
\end{enumerate}
We call the tuple $(G_*,\gamma)$ a {\em folded Morse flow tree}.

\vskip.15in
Given objects $\boldsymbol{f}_0,\boldsymbol{f}_1\in\operatorname{Morse}_\kappa(S)$, we define a grading for $\boldsymbol{q}\in\operatorname{hom}(\boldsymbol{f}_0,\boldsymbol{f}_1)$ by
\begin{equation}
    |\boldsymbol{q}|\coloneqq |q_1|+\dots+|q_\kappa|,
\end{equation}
where $|q_i|$ is the Morse co-index of $q_i$.

Given $\boldsymbol{q}=(q_1,\dots,q_\kappa)\in\operatorname{hom}(\boldsymbol{f}_0,\boldsymbol{f}_1)$ with $q_i\in\mathrm{Crit}(f_{0,i},f_{1,\sigma(i)})$ for some $\sigma\in\mathfrak{S}$, we define the action
\begin{equation}
    \mathcal{A}(\boldsymbol{q})\coloneqq \sum_{i=1}^\kappa f_{1,\sigma(i)}(q_i)-\sum_{i=1}^\kappa f_{0,i}(q_i).
\end{equation}
By definition, if $(G_*,\gamma)\in\mathcal{M}(\boldsymbol{q}_1,\dots,\boldsymbol{q}_m;\boldsymbol{q}_0)$, then
\begin{equation}
    \mathcal{A}(\boldsymbol{q}_0)-\sum_{i=1}^m\mathcal{A}(\boldsymbol{q}_i)=-\sum_{\tilde{e}\in \tilde{E}}\int_{\tilde{e}} \gamma^*(df_{i_r(\tilde{e}),h_r(\tilde{e})}-df_{i_l(\tilde{e}),h_l(\tilde{e})})\leq 0.
\end{equation}
We also write $\mathcal{M}_T^\chi(\boldsymbol{q}_1,\dots,\boldsymbol{q}_m;\boldsymbol{q}_0)$ to specify $T$ and the Euler characteristic $\chi$ of the domain graph $G$.

Note that in this paper, we use $\mathbb{Z}_2\llbracket\hbar\rrbracket$ as the coefficient ring and omit the discussion about orientations of related moduli spaces.

\begin{lemma}
\label{lemma-trans}
    Fixing a generic choice of perturbation data, $\mathcal{M}^\chi(\boldsymbol{q}_1,\dots,\boldsymbol{q}_m;\boldsymbol{q}_0)$ is a smooth manifold of dimension
    \begin{equation*}
        |\boldsymbol{q}_0|-|\boldsymbol{q}_1|-\dots-|\boldsymbol{q}_m|+m-2.
    \end{equation*}
\end{lemma}
\begin{proof}
    This follows from Lemma \ref{lemma-fredholm} and the Sard-Smale Theorem.
\end{proof}

We impose an additional assumption to achieve compactness:
\begin{enumerate}[label=(\Alph*)]
    \item $C^0$-boundedness: There is a compact subset $K\subset S$ such that for each $(G_*,\gamma)\in\mathcal{M}(\boldsymbol{q}_1,\dots,\boldsymbol{q}_m;\boldsymbol{q}_0)$, the image $\gamma(G)\subset K$. \label{condition-A}
\end{enumerate}

\begin{lemma}
\label{lemma-cpt}
    Under assumption \ref{condition-A} and fixing a generic choice of perturbation data, if $|\boldsymbol{q}_0|-|\boldsymbol{q}_1|-\dots-|\boldsymbol{q}_m|+m-2\leq1$, then $\mathcal{M}^\chi(\boldsymbol{q}_1,\dots,\boldsymbol{q}_m;\boldsymbol{q}_0)$ admits a compactification for each $\chi$. 
\end{lemma}
\begin{proof}
    By assumption, $\operatorname{dim}\mathcal{M}^\chi(\boldsymbol{q}_1,\dots,\boldsymbol{q}_m;\boldsymbol{q}_0)\leq 1$.
    Given a sequence
    \begin{equation*}
        \{(G_{*1},\gamma_1),(G_{*2},\gamma_2),\dots\}\subset\mathcal{M}^\chi(\boldsymbol{q}_1,\dots,\boldsymbol{q}_m;\boldsymbol{q}_0),
    \end{equation*}
    we consider all possible degenerations.
    Denote the underlying sequence of folded ribbon trees by $\{(T_i,\sigma_i,l_i)\}$.
    For notation simplicity, suppose we have already taken a subsequence of $\{(G_{*i},\gamma_i)\}$ so that $\{T_i\}$ all have the same topology, i.e., $T_i=T$, and $\{\sigma_i\}$ are all the same.
    We still denote the subsequence by $\{(G_{*i},\gamma_i)\}$. 

    \begin{figure}[ht]
        \centering
        \includegraphics[width=12cm]{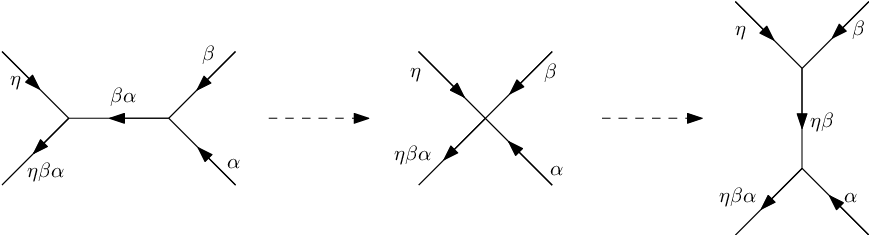}
        \caption{As the length of an inner edge approaches 0, the family of graphs can be continued. The permutations are labeled.}
        \label{fig-cpt-inner}
    \end{figure}

    By Lemma \ref{lemma-no-diagonal}, there can be at most one edge whose length approaches 0 as $i\to\infty$. 
    Suppose this edge exists and is denoted by $e$.

    If $l_i(e)\to 0$ as $i\to\infty$ for some inner edge $e\in E_{inn}$, we refer to Figure \ref{fig-cpt-inner} for this case, where the horizontal edge on the left side represents $e$. 
    Denote the vertical edge on the right side by $e'$. 
    If $e'\in E_{ste}$ or $\sigma_i(e')=\eta\beta\neq1$, then the right side continues the family on the left side. 
    If $e'\in E_{mar}$ and $\sigma_i(e')=\eta\beta=1$, and the left family degenerates to the middle one, then since no adjacent vertex of $E_{\succeq e'}$ belongs to $V_{ext}$,
    \begin{equation}
        A_0\coloneqq\sum_{\pi(\tilde{e})\in E_{\succeq e'}}\int_{\tilde{e}} \gamma^*(df_{i_r(\tilde{e}),h_r(\tilde{e})}-df_{i_l(\tilde{e}),h_l(\tilde{e})})=0,
    \end{equation}
    as the above integral is over closed loops.
    However, $A_0\geq0$ by definition, where the equality holds if and only if $\gamma$ is constant on $\pi^{-1}(E_{\succeq e'})$. 
    Therefore, if $e'\in E_{mar}$ and $\sigma_i(e')=\eta\beta=1$, the left family cannot degenerate to the middle one.

    \begin{figure}[ht]
        \centering
        \includegraphics[width=10cm]{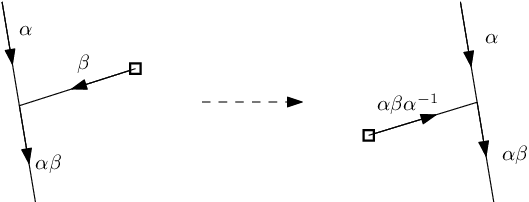}
        \caption{As the length of a marginal exterior edge approaches 0, the family of graphs can be continued. The permutations are labeled.}
        \label{fig-cpt-mar}
    \end{figure}
    
    If $l_i(e)\to 0$ as $i\to\infty$ for some $e\in E_{ext}\cap E_{mar}$, then the degeneration can be continued, where $e$ is transformed to another $e'\in E_{ext}\cap E_{mar}$ and the associated permutation changes by conjugation. See Figure \ref{fig-cpt-mar}. 

    If $l_i(e)\to \infty$ as $i\to\infty$ for some $e\in E_{mar}$, then either $f_{i_r(\tilde{e}),h_r(\tilde{e})}-f_{i_l(\tilde{e}),h_l(\tilde{e})}$ has a critical point where $i_r(\tilde{e})=i_l(\tilde{e})$ and $\tilde{e}\in\pi^{-1}(e)$, or the image $\gamma_i(\tilde{e})\subset S$ is noncompact as $i\to\infty$. 
    The former case is excluded by definition, and the latter case is excluded by assumption \ref{condition-A}.

    If $l_i(e)\to \infty$ as $i\to\infty$ for some $e\in E_{ste}$, then either $e$ limits to two broken edges, or the image $\gamma_i(\tilde{e})\subset S$ is noncompact as $i\to\infty$ for some $\tilde{e}\in\pi^{-1}(e)$. 
    The later case is excluded by assumption \ref{condition-A}.
    Therefore, $e$ limits to two broken edges, which corresponds to the $A_\infty$-relations. 

\end{proof}

Now for each $m\geq1$, we can define the higher composition map by
\begin{gather*}
    \mu^m\colon \operatorname{hom}(\boldsymbol{f}_{m-1},\boldsymbol{f}_{m})\otimes\dots\otimes\operatorname{hom}(\boldsymbol{f}_{0},\boldsymbol{f}_{1})\to\operatorname{hom}(\boldsymbol{f}_{0},\boldsymbol{f}_{m}),\\
    \mu^m(\boldsymbol{q}_m,\dots,\boldsymbol{q}_1)\coloneqq \sum_{\substack{|\boldsymbol{q}_0|=|\boldsymbol{q}_1|+\dots+|\boldsymbol{q}_m|+2-m,\\\chi\leq\kappa}} \#\mathcal{M}^{\chi}(\boldsymbol{q}_1,\dots,\boldsymbol{q}_m;\boldsymbol{q}_0).
\end{gather*}

\begin{proposition}
    Under assumption \ref{condition-A} and fixing a generic choice of perturbation data, $\{\mu^m,m\geq1\}$ defines an $A_\infty$ structure on $\operatorname{Morse}_\kappa(S)$.
\end{proposition}
\begin{proof}
    This follows by Lemma \ref{lemma-trans} and Lemma \ref{lemma-cpt}.
\end{proof}



\begin{figure}[ht]
    \centering
    \includegraphics[width=6cm]{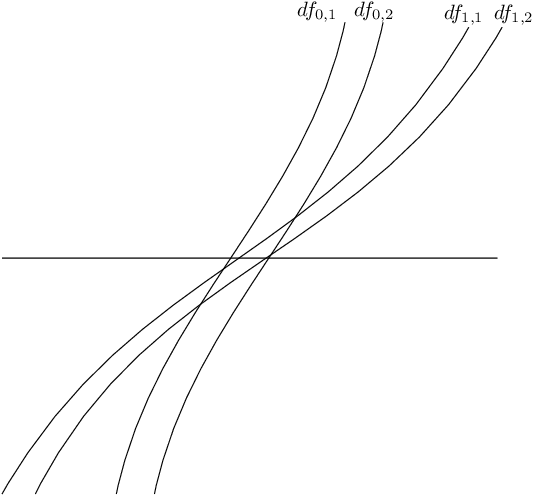}
    \caption{Wrapped cotangent fibers. View $T^*\mathbb{R}^2$ in the projection to $T^*\mathbb{R}_{x_1}$.}
    \label{fig-fiber}
\end{figure}

\subsection{The 2-plane}
\label{section-disk}

From now on, we restrict to $S=\mathbb{R}^2_{x_1,x_2}$. 
Denote by $\mathfrak{M}^R_{\mathbb{R}^2}$ the set of Riemann metrics on $\mathbb{R}^2$ that coincide with the Euclidean metric on $\mathbb{R}^2\backslash D_{R}$ for some large $R>0$, where $D_{R}=\{x_1^2+x_2^2<R^2\}\subset\mathbb{R}^2$.

Pick $\kappa$ disjoint cotangent fibers $\sqcup_{i=1}^\kappa T^*_{q_i}\mathbb{R}^2 \subset T^*\mathbb{R}^2$ where $\{q_1,\dots,q_\kappa\}\subset D_{R/2}$.
By wrapping $\sqcup_{i=1}^\kappa T^*_{q_i}\mathbb{R}^2$ in the positive Reeb direction using some Hamiltonian function which is quadratic at infinity, we get $\boldsymbol{L}_0=\{L_{0,1},\dots,L_{0,\kappa}\}$. 
We denote the positive wrapping by $\sqcup_{i=1}^\kappa T^*_{q_i}\mathbb{R}^2\rightsquigarrow\boldsymbol{L}_0$.
By continuing the wrapping, we get a sequence
\begin{equation}
    \sqcup_{i=1}^\kappa T^*_{q_i}\mathbb{R}^2\rightsquigarrow\boldsymbol{L}_0\rightsquigarrow\boldsymbol{L}_1\rightsquigarrow\boldsymbol{L}_2\rightsquigarrow\dots,
\end{equation}
where each $L_{i,j}$ is required to be graphical, allowing us to view $L_{i,j}$ as the differential $df_{i,j}$.
We further require that 
\begin{enumerate}
    \item[(C1)] The difference $df_{i,j}-df_{i,k}$ is constant on $\mathbb{R}^2\backslash D_{R}$ with respect to the Euclidean metric $g_0$ on $\mathbb{R}^2$;\label{C1}

    \item[(C2)] The absolute angle between $\nabla_{g_0}(f_{i,j}-f_{i',k})$ and $\partial/{\partial r}$ is less than $\pi/4$ on $\mathbb{R}^2\backslash D_{R}$, where $i<i'$;\label{C2}

    \item[(C3)] $||\nabla_{g_0}(f_{i,j}-f_{i',k})||\gg||\nabla_{g_0}(f_{i'',r}-f_{i'',s})||$ for all $i\neq i'$ on $\mathbb{R}^2\backslash D_{R}$.\label{C3}
\end{enumerate}
The above conditions can be easily achieved by modifying the wrapping functions.

Let $\boldsymbol{q}_i\in \operatorname{hom}(\boldsymbol{f}_{i-1},\boldsymbol{f}_{i})$ for $i=1,\dots,m$ and $\boldsymbol{q}_0\in \operatorname{hom}(\boldsymbol{f}_{0},\boldsymbol{f}_{m})$, where $\boldsymbol{f}_i=(f_{i,1},\dots,f_{i,\kappa})$.

\begin{lemma}
\label{lemma-disk}
    Fixing a generic choice of perturbation data, $\mathcal{M}(\boldsymbol{q}_1,\dots,\boldsymbol{q}_m;\boldsymbol{q}_0)$ is a smooth manifold of dimension
    \begin{equation*}
        |\boldsymbol{q}_0|-|\boldsymbol{q}_1|-\dots-|\boldsymbol{q}_m|+m-2.
    \end{equation*}
    When the dimension is less or equal to $1$, $\mathcal{M}(\boldsymbol{q}_1,\dots,\boldsymbol{q}_m;\boldsymbol{q}_0)$ admits a compactification.
\end{lemma}
\begin{proof}
    By Lemma \ref{lemma-trans} and Lemma \ref{lemma-cpt}, it suffices to check that assumption \ref{condition-A} is satisfied.

    \begin{figure}[ht]
        \centering
        \includegraphics[width=5cm]{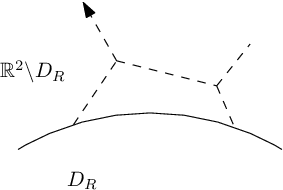}
        \caption{If the image of $\tilde{E}_{mar}$ meets $\mathbb{R}^2\backslash \overline{D}_{R}$, then it cannot be bounded.}
        \label{fig-cpt-disk}
    \end{figure}
    
    Suppose $(G_*,\gamma)\in\mathcal{M}(\boldsymbol{q}_1,\dots,\boldsymbol{q}_m,\boldsymbol{q}_0)$.
    We claim that $\gamma(G)\cap \mathbb{R}^2\backslash \overline{D}_{R}=\emptyset$.
    If not, let $r\coloneqq\operatorname{max}_{x\in G}||\gamma(x)||>R$, where $||\cdot||$ is the Euclidean norm on $\mathbb{R}^2$. 
    There exists $p\in G$ such that $||\gamma(p)||=r$.
    
    If $p\in \operatorname{int}(\tilde{e})$ for $\tilde{e}\in\tilde{E}_{ste}$, then by (C2), we have $\operatorname{max}_{x\in \tilde{e}}||\gamma(x)||>r$.

    If $p\in\tilde{e}$ for $\tilde{e}\in\operatorname{int}(\tilde{E}_{mar})$, then by (C1), we have $\operatorname{max}_{x\in \tilde{e}}||\gamma(x)||>r$.

    If $p\in\tilde{V}$, since the edges adjacent to $p$ satisfy the Kirchhoff laws, there must be one of these edges $\tilde{e}$ such that $\operatorname{max}_{x\in \tilde{e}}||\gamma(x)||>r$.

    All of these cases contradict the definition of $r$.
    Therefore, $\gamma(G)\subset D_{R}$, and hence assumption \ref{condition-A} is satisfied.



\end{proof}

Note that we can arrange the Lagrangians so that $|\boldsymbol{q}|=0$ for all $\boldsymbol{q}\in\operatorname{hom}(\boldsymbol{f}_i,\boldsymbol{f}_j)$, $i<j$. 
In this case, $\mathcal{M}(\boldsymbol{q}_1,\dots,\boldsymbol{q}_m;\boldsymbol{q}_0)$ is of dimension 0 if and only if $m=2$.
It follows that
\begin{corollary}
    The product map
    \begin{equation*}
        \mu^2\colon \operatorname{hom}(\boldsymbol{f}_1,\boldsymbol{f}_2)\otimes \operatorname{hom}(\boldsymbol{f}_0,\boldsymbol{f}_1)\to\operatorname{hom}(\boldsymbol{f}_0,\boldsymbol{f}_2)
    \end{equation*}
    defines an ordinary algebra $\operatorname{End}_{Mor}(\sqcup_{i=1}^\kappa T^*_{q_i}\mathbb{R}^2)$.
\end{corollary}

\section{Equivalence to higher-dimensional Heegaard Floer homology}

\subsection{Review of higher-dimensional Heegaard Floer homology (HDHF)}
\label{section-hdhf}

Let $(X,\alpha)$ be a $2n$-dimensional completed Liouville domain and $\omega=d\alpha$ be the exact symplectic form on $X$.
The objects of the $A_\infty$-category $\mathcal{F}_\kappa(X)$ are $\kappa$-tuples of disjoint exact Lagrangians.
Given two objects $\boldsymbol{L}_i=(L_{i,1},\dots,L_{i,\kappa})$, $i=0,1$, whose components are mutually transverse, the morphism $\mathrm{hom}_{\mathcal{F}_\kappa(X)}(\boldsymbol{L}_0,\boldsymbol{L}_1)=CF(\boldsymbol{L}_0,\boldsymbol{L}_1)$ is the free abelian group generated by all $\mathbf{y}=\{y_{1},\dots,y_{\kappa}\}$ where $y_{j}\in L_{0,j}\cap L_{1,\sigma(j)}$ and $\sigma$ is some permutation of $\{1,\dots,\kappa\}$. 
The coefficient ring is set to be $\mathbb{Z}_2\llbracket\hbar\rrbracket$. 

Let $D$ be the unit disk in $\mathbb{C}$ and $D_m=D-\{p_0,\dots,p_m\}$, where $p_i\in\partial D$ are boundary marked points arranged counterclockwise. Let $\partial_i D_m$ be the boundary arc from $p_i$ to $p_{i+1}$. 
Let $\mathcal{A}_m$ be the moduli space of $D_m$ modulo automorphisms; we choose representatives $D_m$ of equivalence classes of $\mathcal{A}_m$ in a smooth manner (e.g., by setting $p_0=-i$ and $p_1=i$) and abuse notation by writing $D_m\in \mathcal{A}_m$.
We call $D_m$ the ``$A_\infty$ base direction''.
See Figure \ref{fig-base}.

\begin{figure}[ht]
    \centering
    \includegraphics[width=5cm]{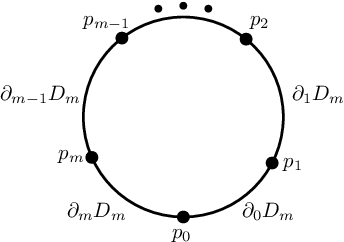}
    \caption{}
    \label{fig-base}
\end{figure}

The full ambient symplectic manifold is then $(D_m\times X,\,\Omega_m=\omega_m+\omega)$, where $\omega_m$ is an area form on $D_m$ which restricts to $ds_i\wedge dt_i$ on the strip-like end $e_i\simeq [0,\infty)_{s_i}\times[0,1]_{t_i}$ around $p_i$ for $i=1,\dots,m$ and $e_i\simeq (-\infty,0]_{s_i}\times[0,1]_{t_i}$ around $p_i$ for $i=0$. (As we approach the puncture $p_i$, $s_i\to +\infty$ for $i=1,\dots,m$ and $s_i\to -\infty$ for $i=0$.) Then we extend the Lagrangians to the $A_\infty$ base direction: for $i=0,\dots,m$, let $\tilde{L}_i=\partial_iD_m\times L_i$ and $\tilde{L}_{ij}=\partial_iD_m\times L_{ij}$. 
Let $\pi_X: D_m\times X\to X$ be the projection to $X$ and 
$\pi_{D_m}$ be the symplectic fibration
\begin{equation*}
    \pi_{D_m}\colon (D_m\times X,\Omega_m)\to (D_m,\omega_m).
\end{equation*}

Let $\mathcal{J}_{X,\alpha}$ be the set of $d\alpha$-compatible almost complex structures $J_X$ on $(X,\omega)$ that are asymptotic to an almost complex structure on $[0,\infty)_s\times \partial X^c$ that takes $\partial_s$ to the Reeb vector field of $\alpha|_{\partial X^c}$, takes $\ker \alpha|_{\partial X^c}$ to itself, and is compatible with $d\alpha|_{\partial X^c}$. 

\begin{definition}
\label{definition-floer-data}
    There is a smooth assignment 
    $D_m\mapsto J_{D_m}$, where $D_m\in \mathcal{A}_m$, such that:
    \begin{enumerate}
        \item[(J1)] on each fiber $\pi_{D_m}^{-1}(p)=\{p\}\times X$, $J_{D_m}$ restricts to an element of $\mathcal{J}_{X,\alpha}$;
        \item[(J2)] $J_{D_m}$ projects holomorphically onto $D_m$;
        \item[(J3)] over each strip-like end $[0,\infty)_{s_i}\times[0,1]_{t_i}$, for $s_i$ sufficiently positive, or over $(-\infty,0]_{s_0}\times[0,1]_{t_0}$, for $s_0$ sufficiently negative, $J_{D_m}$ is invariant in the $s_i$-direction and takes $\partial_{s_i}$ to $\partial_{t_i}$; when $m=1$, $J_{D_1}$ is invariant under $\mathbb{R}$-translation of the base and takes $\partial_{s_i}$ to $\partial_{t_i}$.
    \end{enumerate}
    One can inductively construct such an assignment for all $m\geq1$ in a manner which is (A) consistent with the boundary strata 
    and (B) for which all the moduli spaces $\mathcal{R}({\bf y}_1,\dots, {\bf y}_m; {\bf y}_0)$, defined below, are transversely cut out. 
    A collection $\{ J_{D_m}~|~D_m\in \mathcal{A}_m,~m\in \mathbb{Z}_{>0}\}$ satisfying (A) will be called a {\em consistent collection} of almost complex structures; if it satisfies (B) in addition, it is a {\em sufficiently generic} consistent collection.
\end{definition}


Let $\mathcal{R}(\mathbf{y}_1,\dots,\mathbf{y}_m;\mathbf{y}_0)$ be the moduli space of maps
\begin{equation*}
    u\colon (\dot F,j)\to(D_m\times X,J_{D_m}),
\end{equation*}
where $(F,j)$ is a compact Riemann surface with boundary, $\mathbf{p}_0,\dots,\mathbf{p}_m$ are disjoint $\kappa$-tuples of boundary marked points of $F$, $\dot F=F\setminus\cup_i \mathbf{p}_i$, and $D_m\in \mathcal{A}_m$, so that $u$ satisfies
\begin{align}
    \label{floer-condition}
    \left\{
        \begin{array}{ll}
            \text{$du\circ j=J_{D_m}\circ du$;}\\
            \text{each component of $\partial \dot F$ is mapped to a unique $\tilde{L}_{ij}$;}\\
            \text{$\pi_X\circ u$ tends to $\mathbf{y}_i$ as $s_i\to+\infty$ for $i=1,\dots,m$;}\\
            \text{$\pi_X\circ u$ tends to $\mathbf{y}_0$ as $s_0\to-\infty$;}\\
            \text{$\pi_{D_m}\circ u$ is a $\kappa$-fold branched cover of $D_m$.}
        \end{array}
    \right.
\end{align}
With the identification of the strip-like end $e_i$, $i=1,\dots,m$, as $[0,\infty)_{s_i}\times[0,1]_{t_i}$, the third condition means that $u$ maps the neighborhoods of the punctures of $\mathbf{p}_i$ asymptotically to the Reeb chords $[0,1]_{t_i}\times \mathbf{y}_i$ as $s_i\to +\infty$. The fourth condition is similar. 

The $\mu^m$-composition map of $\mathcal{F}_\kappa(X)$ is then defined as
\begin{gather*}
\label{eq-m_2}
    \mu^m\colon \operatorname{hom}(\boldsymbol{L}_{m-1},\boldsymbol{L}_{m})\otimes\dots\otimes\operatorname{hom}(\boldsymbol{L}_{0},\boldsymbol{L}_{1})\to\operatorname{hom}(\boldsymbol{L}_{0},\boldsymbol{L}_{m}),\\
    \mu^m(\mathbf{y}_m,\dots,\mathbf{y}_1)=\sum_{\mathbf{y}_0,\chi\leq\kappa}\#\mathcal{R}^{\mathrm{ind}=0,\chi}(\mathbf{y}_1,\dots,\mathbf{y}_m;\mathbf{y}_0)\cdot\hbar^{\kappa-\chi}\cdot\mathbf{y}_0,
\end{gather*}
where the superscript $\chi$ denotes the Euler characteristic of $F$; the symbol $\#$ denotes the signed count of the corresponding moduli space.

\begin{proposition}
\label{prop-ind}
    The Fredholm index of $\mathcal{R}^\chi(\mathbf{y}_1,\dots,\mathbf{y}_m;\mathbf{y}_0)$ is
    \begin{equation}
        \label{eq-ind}
        \mathrm{ind}(u)=(n-2)\chi+\mu+2\kappa-m\kappa n+m-2,
    \end{equation}
    where $\mu$ is the Maslov index of $u$, defined as in \cite[Section 4]{colin2020applications}.
\end{proposition}

If $2c_1(TX)=0$ and the Maslov classes of all involved Lagrangians vanish, then there exists a well-defined $\mathbb{Z}$-grading. In this case, the dimension of $\mathcal{R}^{\chi}(\mathbf{y}_1,\dots,\mathbf{y}_m;\mathbf{y}_0)$ can be rewritten as
\begin{equation}
    \label{eq-grading}
    \mathrm{ind}(u)=(n-2)(\chi-\kappa)+|\mathbf{y}_0|-|\mathbf{y}_1|-\dots-|\mathbf{y}_m|+m-2,
\end{equation}
where $|\mathbf{y}_i|=|y_{i1}|+\dots+|y_{i\kappa}|$. Note also that 
\begin{equation}
    \label{eq-hgrading}
    |\hbar|=2-n.
\end{equation}

We omit the details about the orientation of $\mathcal{R}(\mathbf{y}_1,\dots,\mathbf{y}_m;\mathbf{y}_0)$ and the $A_\infty$-relations, and refer the reader to \cite[Section 4]{colin2020applications} for further information.

\subsection{Mushroom}

The construction of ``mushroom'' in this section generalizes the concept introduced in \cite[Section 4]{abouzaid2011topological} for $\kappa\geq1$. 
Roughly speaking, a mushroom combines the $A_\infty$ base direction in the HDHF moduli space with the folded ribbon tree in the Morse moduli space.

\begin{figure}[ht]
    \centering
    \includegraphics[width=10cm]{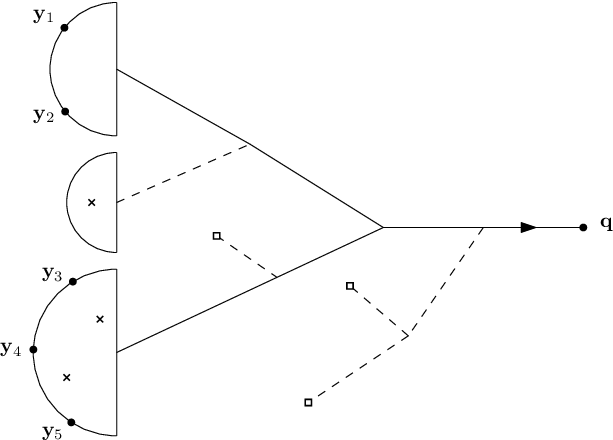}
    \caption{}
    \label{fig-mushroom}
\end{figure}

\begin{definition}
\label{definition-mushroom}
    A mushroom is a tuple $P=(\dot{F},\pi_0,\dot{C},T,\sigma,l)$ which satisfies
    \begin{enumerate}
        \item $C$ is a Riemann surface with boundary composed of $w$ disks $C_1,\dots,C_w$. 
        $\dot{C}_i= C_i\backslash(\sqcup_{j=1}^{m_i} p_{ij}\cup a_i \cup b_i)$ where $a_i,p_{i1},\dots,p_{im_i},b_i$ appear in counterclockwise order on $\partial C_i$ and $m_1+\dots+m_w=m$.
        Denote the arc on $\partial \dot{C}_i$ from $b_i$ to $a_i$ by $\alpha_i$.
        We call $C_i$ a {\em stem disk} if $m_i\geq1$ and a {\em marginal disk} if $m_i=0$.
        We rename $(p_{11},\dots,p_{1m_1},\dots,p_{wm_1},\dots,p_{wm_w})$ as $(p_1,p_2,\dots,p_m)$ and call them incoming boundary punctures.
        We choose strip-like ends near each boundary puncture of $\dot{C}$, similar to Section \ref{section-hdhf}.
        We let $\{p_i\}$ and $\{a_i\}$ be positive punctures and let $\{b_i\}$ be negative punctures.

        \item $F$ is a Riemann surface with boundary. $\dot{F}= F\backslash \left((\sqcup_i \mathbf{p}_i)\cup(\sqcup_i\mathbf{a}_i)\cup(\sqcup_i\mathbf{b}_i)\right)$, where $\mathbf{p}_1,\dots,\mathbf{p}_m$, $\mathbf{a}_1,\dots,\mathbf{a}_w$, $\mathbf{b}_1,\dots,\mathbf{b}_w$ are disjoint $\kappa$-tuples of boundary punctures of $F$.
        $\pi_0\colon \dot{F}\to\dot{C}$ is a $\kappa$-fold branched cover over each $C_i$, mapping $\mathbf{p}_i$ to $p_i$, $\mathbf{a}_i$ to $a_i$, and $\mathbf{b}_i$ to $b_i$.
        We view the branched values of $\pi_0$ on $C_i$ as unlabeled marked points.
        When some marked points coincide, we count their multiplicity.

        \item For each component of the arc $\beta\subset \partial \dot{C}\backslash\sqcup_i\alpha_i$, $\pi_0^{-1}(\beta)\subset \partial \dot{F}$ consists of $\kappa$ disjoint arcs, the set denoted as $S_{\beta}$, and we label them by a bijection $h_{\beta}\colon S_{\beta}\to\{1,\dots,\kappa\}$.
        When there is no ambiguity, we omit the subscript and write $h_\beta=h$.
        
        \item $(T,\sigma,l)$ is a modified folded ribbon tree, where all properties of folded ribbon trees hold, except that now $E_{ext}$ consists of $w$ incoming edges, all of finite length.
        Moreover, for $e\in E_{ext}\cap E_{mar}$, we do not require $\sigma(e)$ to be a transposition.
        We label the $w$ input vertices as $z_{1},\dots,z_{w}$, arranged in counterclockwise order along $\partial D$ when $T$ is embedded in $D$.


        \item For each inner vertex $v\in V_{inn}$, 
        \begin{equation*}
            \sum_{e\in E_{z_{i}\to v}}l(e)=\sum_{e\in E_{z_{j}\to v}}l(e)
        \end{equation*}
        for all $i,j\in\{1,\dots,w\}$.

        \item {\em Compatibility of permutation}. Let $\alpha_{li}, \alpha_{ri}\in \partial\dot{C}_i$ be the edges adjacent to $\alpha_i$ at $a_i,b_i$, respectively.
        The preimage of $\alpha_i$ on $\partial \dot{F}$ consists of $\kappa$ arcs $\alpha_{i1},\dots,\alpha_{i\kappa}$. Denote the adjacent arcs to $\alpha_{ij}$ over $\alpha_{li},\alpha_{ri}$ by $\alpha_{lij},\alpha_{rij}$, respectively. 
        We define a permutation element $\sigma_i\in \mathfrak{S}_\kappa$ such that $h(\alpha_{rij})=\sigma_i(h(\alpha_{lij}))$ for $j=1,\dots,\kappa$. 
        Denote the unique edge adjacent to $z_i$ by $e_{z_i}$. 
        We require that $\sigma(e_{z_i})=\sigma_i$. 
        
        

    \end{enumerate}
\end{definition}

Let $\mathcal{P}^{m,r}$ denote the set of mushrooms with $m$ incoming boundary punctures on $\dot{C}$, $r_1$ branch points on $\dot{C}$, and $r_2$ marginal vertices on $T$, such that $r=r_1+r_2$.
See Figure \ref{fig-mushroom} for an example of a mushroom in $\mathcal{P}^{5,6}$.

\subsection{Perturbation data}
\label{section-mix-data}

From now on, we focus on the 2-plane $S=\mathbb{R}^2$ and consider perturbation data of the mixed moduli space that depends on the Morse functions.
Let $g_0$ be the Euclidean metric on $\mathbb{R}^2$. 
Using $g_0$, we view $T^*\mathbb{R}^2\simeq\mathbb{C}^2$ with the standard complex structure $J_0$.

Consider $m+1$ objects $\boldsymbol{f}_0,\dots,\boldsymbol{f}_m$ of $\operatorname{Morse}_\kappa(\mathbb{R}^2)$ as defined in Section \ref{section-disk}.
In addition to conditions (C1)-(C3), we further require that:
\begin{enumerate}
    \item[(C4)] There exist constants $B_{ij},\theta_{i1},\theta_{i2}\in\mathbb{R}$ such that
    \begin{equation*}
        f_{i,j}=(\kappa-i)(x_1^2+x_2^2) + B_{ij}(\theta_{i1} x_1+\theta_{i2} x_2),
    \end{equation*}
    on $\mathbb{R}^2\backslash D_{R}$, where $|B_{ij}|<1$ and $\theta_{i1}^2+\theta_{i2}^2=1$. 
    Note that this condition is stronger than (C1)-(C3).
\end{enumerate}
Under the above condition, we can express
\begin{equation*}
    f_{i,j}=f_{i,j,1}(x_1)+f_{i,j,2}(x_2),
\end{equation*}
For later use we require $f_{i,j,1}(x_1),f_{i,j,2}(x_2)$ to be extended to be defined on $\mathbb{R}^2$ and denote $F_{i,j}=f_{i,j,1}(x_1)+f_{i,j,2}(x_2)$ on $\mathbb{R}^2$.

Moreover, we can also write
\begin{equation}
\label{eq-split}
    f_{i,j}=\lambda_{i,j,1}(x_{i1})+\lambda_{i,j,2}(x_{i2}),
\end{equation}
where
\begin{gather*}
    \lambda_{i,j,1}(x)=(\kappa-i)x^2+B_{ij}x,\quad\lambda_{i,j,2}(x)=(\kappa-i)x^2,\\
    x_{i1}=\theta_{i1}x_1+\theta_{i2}x_2,\quad x_{i2}=\theta_{i2} x_1-\theta_{i1}x_2.
\end{gather*}
The split form of $f_{i,j}$ in (\ref{eq-split}) will be used to prove compactness in Lemma \ref{lemma-bound}.

As the bridge between Floer and Morse theory, we need a choice of perturbation data associated with each mushroom so that it is compatible with the domain-dependent almost complex structure (Definition \ref{definition-floer-data}) and the domain-dependent metrics (Definition \ref{definition-morse-datum}) on boundary strata.
Denote by $\mathcal{J}^R(T^*\mathbb{R}^2)$ the set of almost complex structures on $T^*\mathbb{R}^2$ that are adapted to the canonical symplectic form and equal to $J_0$ on $T^*(\mathbb{R} ^2\backslash D_R)$.
The perturbed almost complex structure will be parametrized by the branched cover $\pi_0\colon\dot{F}\to \dot{C}$ instead of just $\dot{C}$.


\begin{definition}
    \label{definition-floer-morse-data} 
    A {\bf perturbation datum} on $P=(\dot{F},\pi_0,\dot{C},T,\sigma,l)\in\mathcal{P}^{m,r}$ consists of 
    \begin{enumerate}
        \item a map
        \begin{equation*}
            J_{\pi_0}\colon \dot{F}\to \mathcal{J}^R(T^*\mathbb{R}^2)
        \end{equation*}
        such that:
        \begin{enumerate}
            
            
            \item[(M1)] Over each strip-like end $[0,\infty)_{s_j}\times[0,1]_{t_j}$ (resp. $(-\infty,0]_{s_j}\times[0,1]_{t_j}$) of $\pi_0^{-1}(\dot{C})$, for $s_j$ sufficiently positive (resp. negative), $J_{\pi_0}$ is invariant in the $s_j$-direction and takes $\partial_{s_j}$ to $\partial_{t_j}$;
            
            \item[(M2)] If $C_i$ has $m_i=0$ and has 1 marked point counting multiplicity, then $\sigma_i$ defined in Definition \ref{definition-mushroom} (6) is a transposition. 
            Now $F_i$ is a quadrilateral, viewed as
            \begin{equation*}
                F_i=\{0\leq\operatorname{Re}z\leq d,0\leq\operatorname{Im}z\leq1\}\subset\mathbb{C},
            \end{equation*}
            where $d$ is the modulus, $\pi_0^{-1}(\alpha_i)=F_i\cap\{\operatorname{Re}z\in\{0,d\}\}$, and $\pi_0^{-1}(\alpha_{li})=F_i\cap\{\operatorname{Im}z\in\{0,1\}\}$.
            Suppose $\beta_{i0}\in\pi_0^{-1}(\alpha_{li})$ has $\operatorname{Im}z=0$ and $\beta_{i1}\in\pi_0^{-1}(\alpha_{li})$ has $\operatorname{Im}z=1$.

            Let $k=m_1+\dots+m_{i-1}$.
            Denote by $\phi_0^t, \phi_1^t$ the time-$t$ Hamiltonian flow of 
            \begin{gather*}
                H_0(q,p)=f_{k,h(\beta_{i0})}(q)-F_{k,h(\beta_{i0})}(q)\\
                H_1(q,p)=f_{k,h(\beta_{i1})}(q)-F_{k,h(\beta_{i1})}(q)
            \end{gather*}
            with respect to the canonical symplectic form on $T^*\mathbb{R}^2$.
            We specify $J_{\pi_0}$ as $J(t)$, where $t=\operatorname{Im}z\in[0,1]$, so that
            \begin{equation*}
                J(t)=(\phi_1^t\circ\phi_0^{1-t})_*\circ J_0\circ(\phi_1^t\circ\phi_0^{1-t})^{-1}_*.
            \end{equation*}
            Note that there are two ways to identify $F_i$ with $\{0\leq\operatorname{Re}z\leq d,0\leq\operatorname{Im}z\leq1\}$. However, by the symmetry of construction, the two ways define the same $J(t)$.
        \end{enumerate}

        \item a map
        \begin{equation*}
            g\colon T\to \mathfrak{M}^R_{\mathbb{R}^2},
        \end{equation*}
        which depends continuously on $(T,l)$.
    \end{enumerate}
    We require the compatibility with boundary strata, i.e., when elements in $\mathcal{P}^{m,r}$ limit to 
    \begin{equation*}
        \mathcal{A}^{m_1,r_1}\times \mathcal{P}^{m_2,r_2},\quad m_1+m_2=m+1,\,r_1+r_2=r,
    \end{equation*}
    as shown on the right of Figure \ref{fig-cpt-F3}, or
    \begin{equation*}
        \mathcal{P}^{m_1,r_1}\times\mathcal{T}^{m_2,\chi_2},\quad m_1+m_2=m+1,\,r_1+\kappa-\chi_2=r,
    \end{equation*}
    as shown on the left of Figure \ref{fig-cpt-F3}, we need conditions similar to Definition \ref{definition-floer-data} and Definition \ref{definition-morse-datum}.
    The details are omitted here.

\end{definition}

\begin{lemma}
\label{lemma-floer-morse-data}
    Suppose $\mathcal{P}^{m,r}$ has been chosen for $m\leq m_0$, $r\leq r_0$, and satisfies the above conditions, then it can be extended to a choice of perturbation data for all $m,r\geq0$.
\end{lemma}
\begin{proof}
    The construction is similar to that of Definition \ref{definition-floer-data} and Lemma \ref{lemma-perturbation-data}.
    Note that when $\mathcal{P}^{m,r}$ limits to $\mathcal{A}^{m_1,r_1}\times \mathcal{P}^{m_2,r_2}$, the almost complex structure parametrized by $\dot{F}$ reduces to being parametrized by $D_{m_1}$ on $\mathcal{A}^{m_1,r_1}$ as in Definition \ref{definition-floer-data}.
\end{proof}

\subsection{The $A_\infty$-map from Floer to Morse}

Consider the $m+1$ objects $\boldsymbol{f}_0,\dots,\boldsymbol{f}_m$ of $\operatorname{Morse}_\kappa(T^*\mathbb{R}^2)$ as defined in Section \ref{section-mix-data}. 
Let $\Gamma_{d\boldsymbol{f}_i}\coloneqq\{\Gamma_{df_{i,1}},\dots,\Gamma_{df_{i,\kappa}}\}$ be the $\kappa$-tuple of graphical Lagrangians in $T^*\mathbb{R}^2$ for $i=0,\dots,m$.
Then $\{\Gamma_{d\boldsymbol{f}_i}\}$ are objects of $\mathcal{F}^{gr}_\kappa(T^*\mathbb{R}^2)$.

Let $\mathbf{y}_i\in \mathrm{hom}(\Gamma_{d\boldsymbol{f}_{i-1}},\Gamma_{d\boldsymbol{f}_i})$ for $i=1,\dots,m$ and $\boldsymbol{q}_0\in \mathrm{hom}(\boldsymbol{f}_0,\boldsymbol{f}_m)$.
The moduli space $\mathcal{H}^{r}(\mathbf{y}_1,\dots,\mathbf{y}_m;\boldsymbol{q}_0)$ consists of tuples $(P,u,\gamma)$ where 

\begin{enumerate}
    \item $P=(\dot{F},\pi_0,\dot{C},T,\sigma,l)\in\mathcal{P}^{m,r}$ is a mushroom.
    
    \item The map
    \begin{equation*}
        u\colon (\dot F,j)\to(T^*\mathbb{R}^2,J_{\pi_0}),
    \end{equation*}
    satisfies
    \begin{equation*}
        du\circ j=J_{\pi_0}\circ du,
    \end{equation*}
    where $u$ tends to $\mathbf{y}_i$ as $s_i\to+\infty$ for $i=1,\dots,m$.
    
    \item If $\beta\subset \partial \dot{C}\backslash\sqcup_i\alpha_i$ and $\beta_i\subset\pi_0^{-1}(\beta)\subset \partial \dot{F}$, so that $\beta$ lies between $p_i$ and $p_{i+1}$, and $h(\beta_i)=k$, then $\beta_i$ is mapped to $\Gamma_{df_{i,k}}$ under $u$.

    \item The modified folded ribbon tree $T_*=(T,\sigma,l)$ uniquely determines a modified ribbon graph $G_*=(G,T_*,\pi,h_l,h_r,i_l,i_r)$, where the generalization is straightforward.
    The map $\gamma\colon G\to \mathbb{R}^2$ is a continuous map so that
    \begin{equation*}
        \dot{\gamma}|_{\tilde{e}}=\nabla_g(f_{i_r(\tilde{e}),h_r(\tilde{e})}-f_{i_l(\tilde{e}),h_l(\tilde{e})}),
    \end{equation*}
    for each $\tilde{e}\in \tilde{E}$ identified with an interval of $\mathbb{R}$ specified by the underlying modified folded ribbon tree $\mathcal{T}_*$.

    \item For each $i\in\{1,\dots,w\}$, there exist $\zeta_{ij}\in \mathbb{R}^2$, $j=1,\dots,\kappa$ so that 
    \begin{equation*}
        u(\alpha_{ij})\subset T_{\zeta_{ij}}^*\mathbb{R}^2,
    \end{equation*}
    where $\{\alpha_{ij},\, j=1,\dots,\kappa\}$ are the components of $\pi_0^{-1}(\alpha_i)$.
    Let $\tilde{e}\in\pi^{-1}(e_{z_i})$ so that $h_l(\tilde{e})=h(\beta_{lij})$.
    We then require that $\gamma(v_{st}(\tilde{e}))=\zeta_{ij}$.
    We call $\{\zeta_{ij}\}$ {\em paste points}.
    \label{enum-paste-point}

    \item For $\tilde{v}\in\pi_0^{-1}(v_0)$, $\gamma(\tilde{v})=q_{0,h_l(\tilde{e})}$, where $\tilde{e}$ is the unique edge in $\tilde{E}$ adjacent to $\tilde{v}$.

\end{enumerate}

Given $\mathbf{y}=(y_1,\dots,y_\kappa)\in\operatorname{hom}(\Gamma_{d\boldsymbol{f}_{0}},\Gamma_{d\boldsymbol{f}_1})$ with $y_i\in \Gamma_{df_{0,i}}\cap\Gamma_{df_{1,\sigma(i)}}$ for some $\sigma\in\mathfrak{S}_\kappa$, we define the action
\begin{equation}
    \mathcal{A}(\mathbf{y})\coloneqq\sum_{i=1}^\kappa f_{1,\sigma(i)}(\pi_{\mathbb{R}^2}(q_i))-\sum_{i=1}^\kappa f_{0,i}(\pi_{\mathbb{R}^2}(q_i)),
\end{equation}
where $\pi_{\mathbb{R}^2}\colon T^*\mathbb{R}^2\to\mathbb{R}^2$ is the projection.
By definition, if $(P,u,\gamma)\in\mathcal{H}(\mathbf{y}_1,\dots,\mathbf{y}_m;\boldsymbol{q}_0)$, then
\begin{equation}
\label{eq-F-energy}
    \mathcal{A}(\boldsymbol{q}_0)-\sum_{i=1}^m\mathcal{A}(\mathbf{y}_i)=-\int_{\dot{F}} u^*\omega-\sum_{\tilde{e}\in \tilde{E}}\int_{\tilde{e}} \gamma^*(df_{i_r(\tilde{e}),h_r(\tilde{e})}-df_{i_l(\tilde{e}),h_l(\tilde{e})})\leq 0,
\end{equation}
where $\omega$ is the canonical exact symplectic form on $T^*\mathbb{R}^2$.

\begin{lemma}
\label{lemma-trans-E}
    Fixing a generic choice of perturbation data, $\mathcal{H}^r(\mathbf{y}_1,\dots,\mathbf{y}_m;\boldsymbol{q}_0)$ is a smooth manifold of dimension
    \begin{equation}
    \label{eq-dim-mix}
        |\boldsymbol{q}_0|-|\mathbf{y}_1|-\dots-|\mathbf{y}_m|+m-1.
    \end{equation}
\end{lemma}
\begin{proof}
    The case of $\kappa=1$ is discussed in \cite{abouzaid2011topological}.
    The generalization to $\kappa\geq1$ is similar to Lemma \ref{lemma-trans}.
\end{proof}

The $C^0$-boundedness similar to condition \ref{condition-A} also holds here:
\begin{lemma}
\label{lemma-bound}
    For each $(P,u,\gamma)\in\mathcal{H}^r(\mathbf{y}_1,\dots,\mathbf{y}_m;\boldsymbol{q}_0)$, the image of $\gamma$ lies in $\overline{D}_{R}$ and the image of $u$ lies in $T^*\overline{D}_{R}$. 
\end{lemma}
\begin{proof}
    Let $(P,u,\gamma)\in\mathcal{H}^r(\mathbf{y}_1,\dots,\mathbf{y}_m;\boldsymbol{q}_0)$. Denote 
    \begin{gather*}
        r_C\coloneqq\operatorname{max}_{x\in {F}}||\pi_{\mathbb{R}^2}\circ u(x)||,\quad r_T\coloneqq\operatorname{max}_{x\in {G}}||\gamma(x)||,\quad r_{paste}\coloneqq\operatorname{max}_{i,j}||\zeta_{ij}||,
    \end{gather*}
    where $\{\zeta_{ij}\}$ are the paste points defined in (\ref{enum-paste-point}).

    If $r_C>R$ and $r_C>r_{paste}$, it is not possible by the same arguments as in \cite[Lemma 7.4]{honda2022jems}.

    If $r_T>R$ and $r_T>r_{paste}$, it is also not possible, similar to the proof of Lemma \ref{lemma-disk}.

    If $||\zeta_{ij}||=r_{paste}=r_C=r_T>R$, suppose $\zeta_{ij}=\gamma(v_{st}(\tilde{e}))$ for $\tilde{e}\in\pi^{-1}(e_i)$.
    If $e_i\in E_{ste}$, then by condition (C2), $\operatorname{max}_{x\in \tilde{e}}||\gamma(x)||>r_{paste}$, which contradicts the asssumption.
    If $e_i\in E_{mar}$ and $h_r(\tilde{e})=h_l(\tilde{e})$, then $\gamma|_{\tilde{e}}$ is a constant map, reducing it to the second case.
    If $e_i\in E_{mar}$ and $h_r(\tilde{e})\neq h_l(\tilde{e})$, we need to further investigate the relative position of the image of $u$ and $\gamma$ near $\zeta_{ij}$.

    \begin{figure}[ht]
        \centering
        \includegraphics[width=11cm]{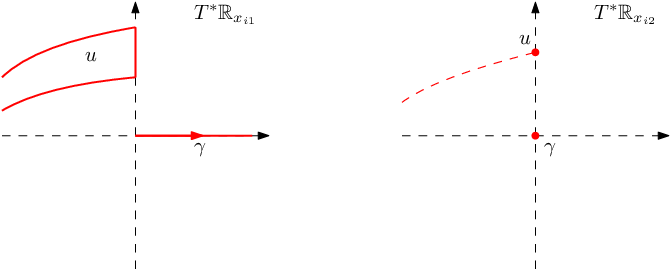}
        \caption{The image of $u$ and $\gamma$ near $T^*_{\zeta_{ij}}\mathbb{R}^2$, where we identify $\mathbb{R}^2$, the target of $\gamma$, with the zero section of $T^*\mathbb{R}^2$. The local coordinates are chosen so that the origin corresponds to $\zeta_{ij}$. The dashed red arc represents the Lagrangian boundary condition in the $T^*\mathbb{R}_{x_{i2}}$-direction. Locally, the projection of $u$ and $\gamma$ to the $T^*\mathbb{R}_{x_{i2}}$-direction must be constant.}
        \label{fig-paste}
    \end{figure}

    Under condition (C4), Figure \ref{fig-paste} shows the local image of $u$ and $\gamma$ near $T^*_{\zeta_{ij}}\mathbb{R}^2$.
    Using the notations defined in (\ref{enum-paste-point}), the image of $\alpha_{ij}$ is represented by the vertical red arc in the $T^*\mathbb{R}_{x_{i1}}$-direction and the constant red dot in the $T^*\mathbb{R}_{x_{i2}}$-direction.
    Since the Lagrangian boundaries in the $T^*\mathbb{R}_{x_{i2}}$-direction is the single dashed arc, the projection of $u$ to $T^*\mathbb{R}_{x_{i2}}$ near $\alpha_{ij}$ must be constant.
    Therefore, the union of the image of $u$ and the image of $\gamma$ contains a straight arc (with respect to $g_0$) passing through $\zeta_{ij}\in\mathbb{R}^2$, hence either $r_C>r_{paste}$ or $r_T>r_{paste}$, which contradicts the assumption.
    
    As a result, $\operatorname{max}\{r_C,r_T\}\leq R$, implying that the image of $\gamma$ lies in $\overline{D}_{R}$ and the image of $u$ lies in $T^*\overline{D}_{R}$.




\end{proof}

\begin{lemma}
\label{lemma-cpt-E}
    Fixing a generic choice of perturbation data, if $|\boldsymbol{q}_0|-|\mathbf{y}_1|-\dots-|\mathbf{y}_m|+m-1\leq 1$, then $\mathcal{H}^r(\mathbf{y}_1,\dots,\mathbf{y}_m;\boldsymbol{q}_0)$ admits a compactification for each $r\geq0$. 
\end{lemma}
\begin{proof}
    Consider a sequence of elements $(u_k,\gamma_k)\in\mathcal{H}^r(\mathbf{y}_1,\dots,\mathbf{y}_m;\boldsymbol{q}_0)$ such that as $k\to\infty$, a branch point on $C_i$ approaches $\alpha_i$ for some $i$.
    There are two cases. 

    \noindent
    (1) On some $C_i$, $m_i>0$ or $C_i$ has more than 1 marked point.
    See Figure \ref{fig-cpt-F} for an example.

    \begin{figure}[ht]
        \centering
        \includegraphics[width=10cm]{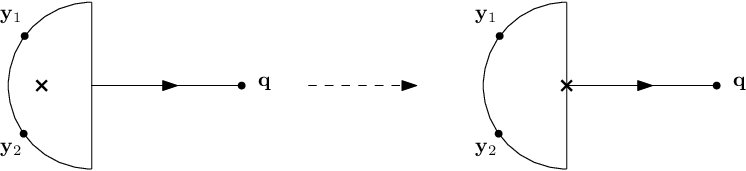}
        \caption{}
        \label{fig-cpt-F}
    \end{figure}
    
    \begin{claim}
        The above degeneration in case (1) is of codimension 2 and cannot occur if $\operatorname{dim}\mathcal{H}^r(\mathbf{y}_1,\dots,\mathbf{y}_m,\boldsymbol{q}_0)\leq 1$.
    \end{claim}
    \begin{proof}
        Consider the moduli space $\mathcal{H}_{res}^r(\mathbf{y}_1,\dots,\mathbf{y}_m;\boldsymbol{q}_0)$ which is defined similarly to $\mathcal{H}^r(\mathbf{y}_1,\dots,\mathbf{y}_m;\boldsymbol{q}_0)$ with only one difference: we replace (6) of Definition \ref{definition-mushroom} by
        \begin{enumerate}
            \item[(6')] Denote by $\beta_{li}, \beta_{ri}\in \partial\dot{C}_i$ the edges adjacent to $\alpha_i$ at $a_i,b_i$, respectively.
            The preimage of $\alpha_i$ on $\partial \dot{F}$ consists of $\kappa$ arcs $\alpha_{i1},\dots,\alpha_{i\kappa}$. 
            Denote the adjacent arcs of $\alpha_{ij}$ over $\alpha_{li},\alpha_{ri}$ by $\alpha_{lij},\alpha_{rij}$, respectively. 
            We define a permutation element $\sigma_i\in \mathfrak{S}_\kappa$ such that $h(\alpha_{rij})=\sigma_i(h(\alpha_{lij}))$ for $j=1,\dots,\kappa$. 
            Denote the unique edge adjacent to $z_i$ by $e_{z_i}$. 
            Then there exists $j$ such that $\sigma(e_{z_i})=\sigma_i$ for $i\neq j$ and $\sigma(e_{z_j})=\sigma_j\tau_{vw}$, where $v\neq w\in\{1,\dots,\kappa\}$ and $\tau_{vw}$ is the transposition of $v,w$.
            Furthermore, we require that 
            \begin{equation}
            \label{eq-same-fiber}
                \gamma(\tilde{v}_{j,v})=\gamma(\tilde{v}_{j,w}),
            \end{equation}
            where $\tilde{v}_{j,b}$ denotes the starting vertex of $\tilde{e}\in\pi^{-1}(e_{z_j})$ with $h_l(\tilde{e})=b$.
        
        \end{enumerate}
        
        Similar to Lemma \ref{lemma-trans-E}, one can show that $\mathcal{H}_{res}^r(\mathbf{y}_1,\dots,\mathbf{y}_m;\boldsymbol{q}_0)$ is a smooth manifold of dimension
        \begin{equation*}
            |\boldsymbol{q}_0|-|\mathbf{y}_1|-\dots-|\mathbf{y}_m|+m-3.
        \end{equation*}
        Roughly speaking, the shift of dimension $2$ compared to (\ref{eq-dim-mix}) comes from the constraint (\ref{eq-same-fiber}).
        Since 
        \begin{equation*}
            \operatorname{dim}\mathcal{H}_{res}(\mathbf{y}_1,\dots,\mathbf{y}_m;\boldsymbol{q}_0)\leq\operatorname{dim}\mathcal{H}(\mathbf{y}_1,\dots,\mathbf{y}_m;\boldsymbol{q}_0)-2\leq-1, 
        \end{equation*}
        we see that $\mathcal{H}_{res}(\mathbf{y}_1,\dots,\mathbf{y}_m,\boldsymbol{q}_0)$ is empty.

        As $k\to\infty$, if a branch point on $C_i$ approaches $\alpha_i$ for some $i$, denote the limit of $(u_k,\gamma_k)$ by $(u_\infty,\gamma_\infty)$.
        Then $u_\infty$ is a nodal curve where the nodal point is mapped to $\alpha_i$ under $\pi_0$.
        After resolving the nodal point, we obtain a normalized curve $u_{norm}$.
        By definition, $(u_{norm},\gamma_\infty)$ is a smooth locus of $\mathcal{H}_{res}^{r-1}(\mathbf{y}_1,\dots,\mathbf{y}_m;\boldsymbol{q}_0)$.
        However, this is not possible.
        
        Therefore, the above degeneration in case (1) is of codimension 2 and cannot occur if $\operatorname{dim}\mathcal{H}^r(\mathbf{y}_1,\dots,\mathbf{y}_m,\boldsymbol{q}_0)\leq 1$.
    \end{proof}

    \noindent
    (2) $m_i=0$ and $C_i$ has only 1 marked point.
    
    In this case, without loss of generality, we assume $\kappa=2$.
    Now $\dot{F}_i=\pi_0^{-1}(\dot{C}_i)$ is a quadrilateral. 
    As the branch point approaches $\alpha_i$, the cross-ratio of $\dot{F}_i$ decreases to $0$ and $u_\infty(\dot{F}_i)=u_\infty(\pi_0^{-1}(\alpha_i))$.
    Let $u_{res}$ be defined by removing the component $\dot{F}_i$ from the domain of $u_\infty$, and let $\gamma_{res}$ be defined by viewing $z_i$ as a marginal vertex on the domain of $\gamma_\infty$.
    It is then easy to check that $(u_{res},\gamma_{res})$ is a smooth locus on the codimension-1 boundary of $\mathcal{H}^r(\mathbf{y}_1,\dots,\mathbf{y}_m;\boldsymbol{q}_0)$. 
    Therefore, the family $\{(u_k,\gamma_k)\}$ can be continued to another family of $\mathcal{H}^r(\mathbf{y}_1,\dots,\mathbf{y}_m;\boldsymbol{q}_0)$.

    \begin{figure}[ht]
        \centering
        \includegraphics[width=14cm]{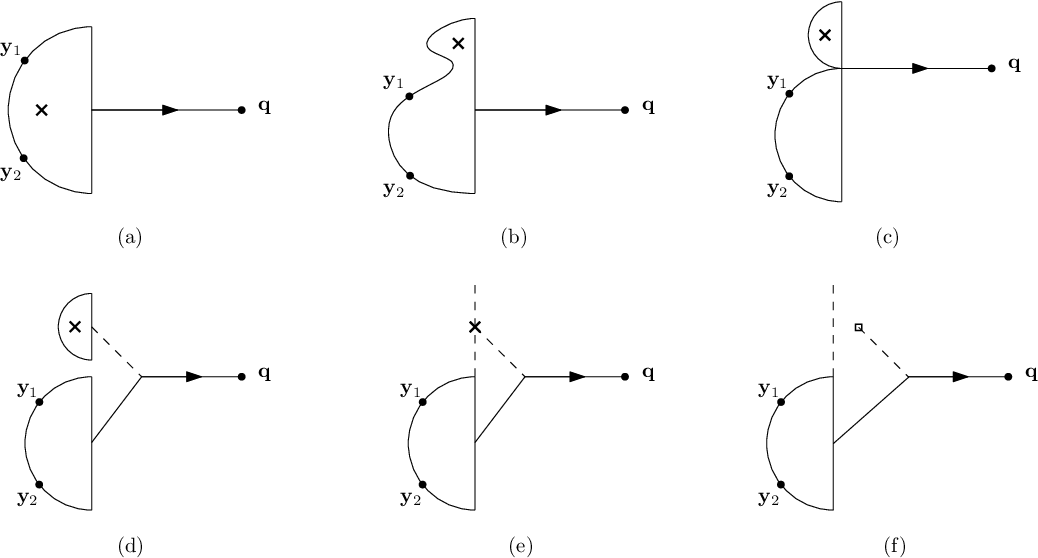}
        \caption{}
        \label{fig-cpt-F1}
    \end{figure}

    \begin{figure}[ht]
        \centering
        \includegraphics[width=10cm]{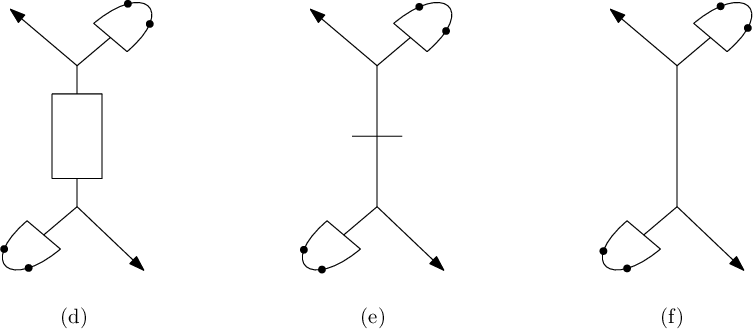}
        \caption{}
        \label{fig-cpt-F2}
    \end{figure}
    
    Figure \ref{fig-cpt-F1} shows an example of such a degeneration, where the stages from (a) to (f) illustrate how a branch point on $\dot{C}_i$ can be transformed to a marginal vertex.
    Case (e) represents the limit as $k\to\infty$, which is a degeneration of codimension 1.
    Figure \ref{fig-cpt-F2} shows what happens on $\dot{F}$ and $G$ from (d) to (f).

    Now consider the reverse procedure, i.e., as the marginal edge in case (f) gets longer and meets (e), we want to continue and reach (d). 
    Due to condition (M2), the quadrilateral curve in (d) equivalently lives in $\mathbb{C}^2$, where the Lagrangian boundaries are products of arcs on each copy of $\mathbb{C}$. 
    This condition ensures that (f) can be uniquely continued to (d), and we omit the details.
    
    \vskip.15in
    By (\ref{eq-F-energy}), the energy $\int_{\dot{F}}u^*\omega$ is bounded, hence the Gromov compactness applies here.
    The remaining types of degenerations are dealt with similarly as in \cite{abouzaid2011topological} and Lemma \ref{lemma-cpt}.

    In conclusion, the only codimension-1 degenerations are the required relations for defining an $A_\infty$-map.
    See Figure \ref{fig-cpt-F3}.
    
\end{proof}

\begin{figure}[ht]
    \centering
    \includegraphics[width=14cm]{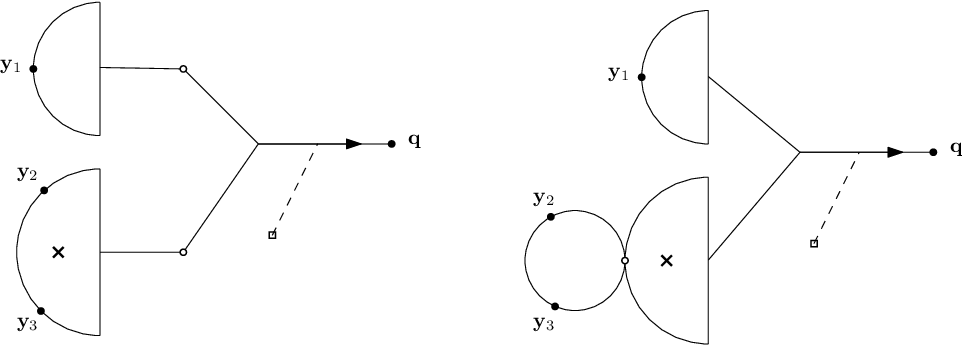}
    \caption{The small circles denote where breakings happen.}
    \label{fig-cpt-F3}
\end{figure}

Consider the functor $\mathcal{E}\colon\mathcal{F}^{gr}_\kappa(T^*\mathbb{R}^2)\to\mathrm{Morse}_{\kappa}(\mathbb{R}^2)$, which maps $\Gamma_{d\boldsymbol{f}}$ to $\boldsymbol{f}$ on the object level.
For each $m\geq1$, we define the map $\mathcal{E}^m$ by
\begin{gather*}
    \mathcal{E}^m\colon \operatorname{hom}(\boldsymbol{f}_{m-1},\boldsymbol{f}_{m})\otimes\dots\otimes\operatorname{hom}(\boldsymbol{f}_{0},\boldsymbol{f}_{1})\to\operatorname{hom}(\Gamma_{\boldsymbol{f}_{0}},\Gamma_{\boldsymbol{f}_{m}}),\\
    \mathcal{E}^m(\mathbf{y}_m,\dots,\mathbf{y}_1)\coloneqq \sum_{\boldsymbol{q}_0,r} \#\mathcal{H}^{r,\operatorname{ind}=0}(\mathbf{y}_1,\dots,\mathbf{y}_m;\boldsymbol{q}_0)\cdot\hbar^{r}\cdot\boldsymbol{q}_0.
\end{gather*}

\begin{proposition}
\label{prop-ainfty}
    $\mathcal{E}$ is an $A_\infty$-functor.
\end{proposition}
\begin{proof}
    This follows from the proof of Lemma \ref{lemma-cpt-E}.
\end{proof}

\subsection{The isomorphism}

\begin{lemma}
\label{lemma-E1}
    $\mathcal{E}^1$ induces an isomorphism on homology.
\end{lemma}
\begin{proof}
    First, we observe that the restriction of $\mathcal{E}^1$ to $\hbar=0$ induces a map on homology:
    \begin{equation*}
        \mathcal{E}^1|_{\hbar=0}\colon \operatorname{hom}(\Gamma_{\boldsymbol{f}_{0}},\Gamma_{\boldsymbol{f}_{1}})|_{\hbar=0}\to \operatorname{hom}(\boldsymbol{f}_{0},\boldsymbol{f}_{1})|_{\hbar=0}.
    \end{equation*}
    This map is an isomorphism by \cite[Theorem 1.1]{abouzaid2011topological} and \cite[Lemma 6.5]{honda2022jems}.
    
    To show the injectivity of $\mathcal{E}^1$, suppose that there exists $\mathbf{a}\in \operatorname{hom}(\Gamma_{\boldsymbol{f}_{0}},\Gamma_{\boldsymbol{f}_{1}})$ such that $\mathcal{E}^1([\mathbf{a}])=0$ and $[\mathbf{a}]\neq0$.
    We can write $\mathbf{a}=\sum_{i\geq0}\hbar^i\mathbf{a}_i$, where $\mathbf{a}_i \in \operatorname{hom}(\Gamma_{\boldsymbol{f}_{0}},\Gamma_{\boldsymbol{f}_{1}})|_{\hbar=0}$. Without loss of generality, we may assume that $[\mathbf{a}_0]\neq0$. 
    Setting $\hbar=0$, we have $\mathcal{E}^1([\mathbf{a}_0])=\mathcal{E}^1([\mathbf{a}])=0$.
    Hence, $\mathcal{E}^1|_{\hbar=0}([\mathbf{a}_0])=0$, which implies $[\mathbf{a}_0]=0$ since $\mathcal{E}|_{\hbar=0}$ is an isomorphism. 
    This leads to a contradiction. 
    Therefore, $\mathcal{E}^1$ is injective.
    
    To prove the surjectivity of $\mathcal{E}^1$, it suffices to show that any $[\mathbf{b}]\in \operatorname{hom}(\boldsymbol{f}_{0},\boldsymbol{f}_{1})$ is in the image of $\mathcal{E}^1$. 
    Since $\mathcal{E}^1|_{\hbar=0}$ is an isomorphism, there exists $\mathbf{a}_0\in \operatorname{hom}(\Gamma_{\boldsymbol{f}_{0}},\Gamma_{\boldsymbol{f}_{1}})|_{\hbar=0}$ such that $\mathcal{E}^1([\mathbf{a}_0])\equiv[\mathbf{b}]\,(\mathrm{mod\,\,\hbar})$. 
    Write $\mathbf{b}-\mathcal{E}^1(\mathbf{a}_0)=\partial A + \hbar B$, and let $\mathbf{b}_1=B|_{\hbar=0}$.
    Then there exists $\mathbf{a}_1\in \operatorname{hom}(\Gamma_{\boldsymbol{f}_{0}},\Gamma_{\boldsymbol{f}_{1}})|_{\hbar=0}$ such that $\mathcal{E}^1([\mathbf{a}_1])\equiv[\mathbf{b}_1]\,(\mathrm{mod\,\,\hbar})$. 
    By repeating this procedure, 
    we get $\mathcal{E}^1(\sum_{i\geq0}[\mathbf{a}_i]\hbar^i)=[\mathbf{b}]$. Hence, $\mathcal{E}^1$ is surjective.

    It then follows that $\mathcal{E}^1$ induces an isomorphism on homology.
\end{proof}

\begin{proposition}
\label{prop-equiv}
    $\mathcal{E}$ is an $A_\infty$-equivalence. 
\end{proposition}
\begin{proof}
    This follows from Proposition \ref{prop-ainfty} and Lemma \ref{lemma-E1}.
\end{proof}

\begin{definition}
    The Hecke algebra $H_\kappa$ is a unital $\mathbb{Z}_2\llbracket\hbar\rrbracket$-algebra generated by $T_1,\dots,T_{\kappa-1}$, with relations
    \begin{align*}
        T_{i}^2&=1+\hbar T_i,\\
        T_iT_j&=T_jT_i\,\,\,\mathrm{for}\,\,\,|i-j|>1,\\
        T_iT_{i+1}T_i&=T_{i+1}T_{i}T_{i+1}.
    \end{align*}
\end{definition}

\begin{theorem}
    $\operatorname{End}_{Mor}(\sqcup_{i=1}^\kappa T^*_{q_i}\mathbb{R}^2)\simeq H_\kappa$.
\end{theorem}
\begin{proof}
    By Proposition \ref{prop-equiv} and \cite[Theorem 1.2]{tian2022example},
    \begin{equation*}
        \operatorname{End}_{Mor}(\sqcup_{i=1}^\kappa T^*_{q_i}\mathbb{R}^2)\simeq\operatorname{End}_{Fuk}(\sqcup_{i=1}^\kappa T^*_{q_i}\mathbb{R}^2)\simeq H_\kappa.
    \end{equation*}
    
\end{proof}

\appendix

\section{Transversality}
\label{section-appendix-trans}

We use the notations from Section \ref{section-morse}.
To prove the transversality of $\mathcal{M}(\boldsymbol{q}_1,\dots,\boldsymbol{q}_m;\boldsymbol{q}_0)$, we follow the similar approach as in \cite[Section 12]{fukaya1997zero}.
For notation convenience, given $(T,\sigma,l)\in\mathcal{T}^{m,\chi}$, we define a map
\begin{equation}
    \rho\colon E\to \mathfrak{S}_\kappa,\quad \rho(e)=\sigma(e)\cdot\sigma^{-1}(e_0).
\end{equation}
For a given $T$, we denote the set of length functions $l$ by $\mathcal{L}_T$.

Let $v^{int}_0=v_{st}(e_0)$ and $v^{int}_{i}=v_{en}(e_i)$ for $i=1,\dots,m$.
It is possible that $v^{int}_i=v^{int}_j$, and we write
\begin{equation*}
    v^{int}_0=\dots=v^{int}_{i_1-1}\neq v^{int}_{i_1}=\dots=v^{int}_{i_2-1}\neq v^{int}_{i_2}=\dots\neq v^{int}_{i_w}=\dots=v^{int}_m.
\end{equation*}
We will define a map
\begin{equation}
    \Psi_{ste}\colon S^\kappa\times \mathcal{L}_T\to S^{\kappa w}.
\end{equation}
For each $1\leq h\leq w$, there is a unique path joining $v^{int}_{i_h}$ and $v^{int}_0$, and the ordered edges along this path are denoted by $e^h_{j_1},\dots,e^h_{j_{c_h}}$.
Define
\begin{equation}
    \Psi_{ste,h}(l)\coloneqq \operatorname{Exp}_{V^h_1}(l(e^h_{j_1}))\circ\dots\circ\operatorname{Exp}_{V^h_{c_h}}(l(e^h_{j_h})),
\end{equation}
where $V^h_b=(V^h_{b,1},\dots,V^h_{b,\kappa})$,
\begin{equation}
\label{eq-flow}
    V^h_{b,i}=\nabla_g\left(f_{i_r(e^h_{j_{b}}),\,\sigma(e^h_{j_{b}})\circ\rho(e^h_{j_{b}})(i)}-f_{i_l(e^h_{j_{b}}),\,\rho(e^h_{j_{b}})(i)}\right),
\end{equation}
and $\operatorname{Exp}_{V}(t)$ denotes the time-$t$ flow of $V$ on $S^\kappa$. 
We then define
\begin{equation*}
    \Psi_{ste}(\boldsymbol{q},l)\coloneqq \left(\Psi_{ste,1}(l)(\boldsymbol{q}),\dots,\Psi_{ste,w}(l)(\boldsymbol{q})\right).
\end{equation*}

\begin{figure}[ht]
    \centering
    \includegraphics[width=12cm]{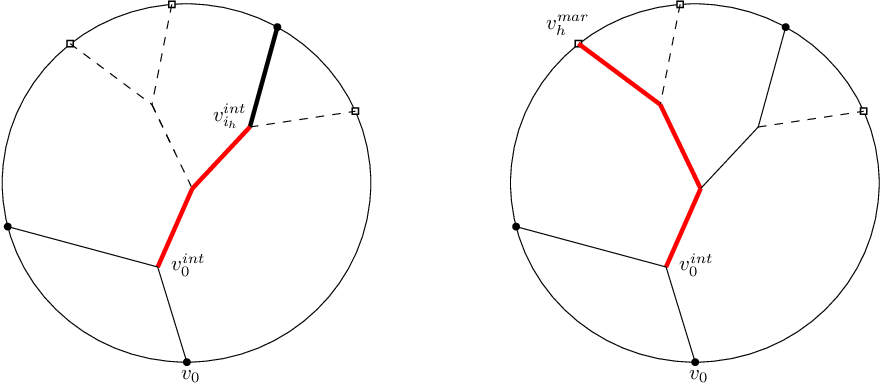}
    \caption{The red path corresponds to $\Psi_{ste,h}$ (resp. $\Psi_{mar,h}$) in the left (resp. right) figure. The perturbation of the metric occurs on all such red arcs.}
    \label{fig-trans}
\end{figure}

Denote the vertices of $V_{mar}$ by $v^{mar}_1,\dots,v^{mar}_\xi$.
For each $1\leq h\leq \xi$, there is a unique path joining $v^{mar}_{h}$ and $v^{int}_0$, and the ordered edges along this path are denoted by $\hat{e}^{h}_{j_1},\dots,\hat{e}^{h}_{j_{d_h}}$.
Define
\begin{equation}
    \Psi_{mar,h}(l)\coloneqq \operatorname{Exp}_{V^h_1}(l(\hat{e}^{h}_{j_1}))\circ\dots\circ\operatorname{Exp}_{V^h_{d_h}}(l(\hat{e}^{h}_{j_h})),
\end{equation}
where $V^h_b$ is defined the same as in (\ref{eq-flow}).
Then we define
\begin{gather*}
    \Psi_{mar}\colon S^\kappa\times \mathcal{L}_T\to S^{\kappa \xi},\\
    \Psi_{mar}(\boldsymbol{q},l)\coloneqq \left(\Psi_{mar,1}(l)(\boldsymbol{q}),\dots,\Psi_{mar,\xi}(l)(\boldsymbol{q})\right).
\end{gather*}
We combine $\Psi_{ste}$ and $\Psi_{mar}$ to obtain the map
\begin{equation*}
    \Psi\colon S^\kappa\times \mathcal{L}_T\to S^{\kappa(w+\xi)},\quad \Psi=(\Psi_{ste},\Psi_{mar}).
\end{equation*}

Let $W^-_{\boldsymbol{q}_{i,\rho(e_i)}}(\boldsymbol{f}_i,\boldsymbol{f}_{i+1})$ be the unstable manifold of the gradient flow of 
\begin{equation*}
    \left(f_{i,\rho(e_i)(1)}-f_{i+1,\sigma(e_i)\circ\rho(e_i)(1)},\dots,f_{i,\rho(e_i)(\kappa)}-f_{i+1,\sigma(e_i)\circ\rho(e_i)(\kappa)}\right)
\end{equation*}
in $S^{\kappa}$, where $\boldsymbol{q}_{i,\rho(e_i)}=(q_{\rho(e_i)(1)},\dots,q_{\rho(e_i)(\kappa)})$.

Denote $\triangle(i,j)=\{(q_1,\dots,q_\kappa)\in S^\kappa\,|\,q_i=q_j\}$, which admits a stratification where the largest stratum is of dimension $2\kappa-2$.

For $e\in E_{mar}\cap E_{ext}$, $\sigma(e)$ is a transposition which exchanges the positions $\alpha(e),\beta(e)$ with $\alpha(e)<\beta(e)$.

The following lemma holds by definition:
\begin{lemma}
    There is a 1-1 correspondence
    \begin{align*}
        \mathcal{M}(\boldsymbol{q}_1,\dots,\boldsymbol{q}_m;\boldsymbol{q}_0)\simeq
        \Psi^{-1}(A).
    \end{align*}
    where
    \begin{equation*}
        A\coloneqq\left(\prod_{h=1}^w\bigcap_{j=i_h}^{i_{h+1}-1} W^-_{\boldsymbol{q}_{j,\rho(e_j)}}(\boldsymbol{f}_j,\boldsymbol{f}_{j+1}),\,\prod_{h=1}^\xi \triangle\left(\rho^{-1}(\hat{e}^h_{j_1})(\alpha(\hat{e}^h_{j_1})),\rho^{-1}(\hat{e}^h_{j_1})(\beta(\hat{e}^h_{j_1}))\right)\right).
    \end{equation*}
\end{lemma}

Denote the set of perturbation data on $T$ by $\mathcal{K}_T$.
To apply the Sard-Smale Theorem, we define another map which allows the variation of the metric:
\begin{gather}
    \widehat{\Psi}\colon S^\kappa\times \mathcal{L}_T\times\mathcal{K}_T\to S^{\kappa(w+\xi)},\quad \widehat{\Psi}(\boldsymbol{q},l,g)=\Psi_g(\boldsymbol{q},l),
\end{gather}
where $\Psi_g$ is the map $\Psi$ with the specified perturbation data $g\colon T\to\mathfrak{M}$.

\begin{lemma}
\label{lemma-fredholm}
    $\widehat{\Psi}^{-1}(A)$ is a smooth submanifold of $S^\kappa\times \mathcal{L}_T\times\mathcal{K}_T$.
    Moreover, the projection $\widehat{\Psi}^{-1}(A)\to\mathcal{K}_T$ is a Fredholm map of index
    \begin{equation*}
        |\boldsymbol{q}_0|-|\boldsymbol{q}_1|-\dots-|\boldsymbol{q}_m|+m-2.
    \end{equation*}
\end{lemma}
\begin{proof}
    The proof is similar to \cite[Lemma 12.6]{fukaya1997zero}. 
    By perturbing the $T$-dependent metric, one can show that the linearization of $\widehat{\Psi}$ at $\widehat{\Psi}^{-1}(A)$ is surjective.
\end{proof}

\section{Compactness}
\label{section-appendix-cpt}

A key difference between Morse flow graphs and Morse flow trees is that for a Morse flow graph, there is an additional type of degeneration: a subgraph may collapse to a constant map.
Such degeneration has been discussed in detail by Fukaya \cite{fukaya1996morse,fukaya1996m} and Watanabe \cite{Watanabe2018}.
We briefly recall their settings and adapt them to our case.

Let $M$ be a manifold of dimension $n$.
Let $G_*=(G,T_*,\pi,h_l,h_r,i_l,i_r)$ be a ribbon graph so that all inner vertices have degree 3. Suppose we have chosen the perturbation data $g\colon G\to\mathfrak{M}$ and have assigned a Morse function to each edge $\tilde{e}\in \tilde{E}$, i.e.,
\begin{equation*}
    F\colon \tilde{E}\to C^\infty(M).
\end{equation*}
Consider the following moduli space:
\begin{equation*}
    \mathcal{M}_{T}(M,F)=\left\{\gamma\colon G\to M\,\left|
    \begin{array}{cc}
        \dot{\gamma}|_{\tilde{e}}=-\nabla_g F(\tilde{e}),  \\
        \gamma\text{ maps $\pm\infty$ to critical points of $F(\tilde{e})$}. 
    \end{array}
    \right.\right\}.
\end{equation*}

Let $\Gamma=(V_\Gamma,E_\Gamma)$ be a connected subtree of $T$ such that $V_\Gamma$ does not contain $\{\pm\infty\}$.
Let $G_\Gamma=(\tilde{V}_\Gamma,\tilde{E}_\Gamma)\coloneqq\pi^{-1}(\Gamma)\subset G$ be the corresponding subgraph of $G$.
We can write $G_\Gamma=G^1_\Gamma\sqcup\dots\sqcup G^r_\Gamma$ as the disjoint union of connected components where $G^i\Gamma=(\tilde{V}^i_\Gamma,\tilde{E}^i_\Gamma)$.

Given $p_i\in M$, consider the real blow-up of $|\tilde{V}^i_\Gamma|$ points at $p_i$:
\begin{equation*}
    C^{local}_{\tilde{V}^i_\Gamma}\coloneqq\{\varphi\colon \tilde{V}^i_\Gamma\to T_{p_i}M\}.
\end{equation*}
We define the linearized flow graph of $\Gamma$ at $\boldsymbol{p}=(p_1,\dots,p_r)\in M^r$ as $\mathcal{M}^{local}_\Gamma(M,F,\boldsymbol{p})$, which consists of $l^{local}\colon E_\Gamma\to\mathbb{R}_+\cup\{0\}$ satisfying the following conditions:
\begin{enumerate}
    \item $\sum_{e\in E_\Gamma}|l^{local}(e)|^2=1$.
    \item For each $i=1,\dots,r$, there exists $\varphi_i\in C^{local}_{\tilde{V}^i_\Gamma}$ such that 
    \begin{equation}
        \label{eq-blow}
        \varphi_i(\tilde{v}_{en}(\tilde{e}))-\varphi_i(\tilde{v}_{st}(\tilde{e}))=\nabla_g F(\tilde{e})\cdot l^{local}(\pi(\tilde{e}))
    \end{equation}
    for all $\tilde{e}\in\tilde{E}^i_\Gamma$.
\end{enumerate}
We also define $\mathcal{M}^{local}_\Gamma(M,F)=\sqcup_{\boldsymbol{p}\in M^r}\mathcal{M}^{local}_\Gamma(M,F,\boldsymbol{p})$, which is viewed as a fibration over $M^r$ with $\mathcal{M}^{local}_\Gamma(M,F,\boldsymbol{p})$ being the fiber over $\boldsymbol{p}\in M^r$.

Consider the moduli space of flow graphs with the collapsed subgraph $G_\Gamma$ defined as
\begin{equation*}
    \mathcal{M}_{T/\Gamma}(M,F)=\mathcal{M}_{T}(M,F)\cap\{l(e)=0\text{ for }e\in\Gamma\}.
\end{equation*}
We also denote 
\begin{equation*}
    \mathcal{M}_{T/\Gamma}(M,F,\boldsymbol{p})\coloneqq\mathcal{M}_{T/\Gamma}(M,F)\cap\{\gamma(G^i_{\Gamma})=\{p_i\},\,i=1,\dots,r\}.
\end{equation*}
Then $\mathcal{M}_{T/\Gamma}(M,F)=\sqcup_{\boldsymbol{p}\in M^r}\mathcal{M}_{T/\Gamma}(M,F,\boldsymbol{p})$, which is also viewed as a fibration over $M^r$.
Taking into account the linearized constraint, we define 
\begin{equation*}
    \mathcal{M}_{T,\Gamma}(M,F)=\mathcal{M}_{T/\Gamma}(M,F)\times_{M^r}\mathcal{M}^{local}_\Gamma(M,F).
\end{equation*}

Similar to Lemma \ref{lemma-trans}, we can show:
\begin{lemma}
    Fixing a generic choice of of $F$ and perturbation data $g\colon G\to\mathfrak{M}$, $\mathcal{M}_{T/\Gamma}(M,F)$ is a smooth manifold.
\end{lemma}

The following lemma is a direct generalization of \cite[Lemma 2.8]{Watanabe2018}:
\begin{lemma}
    Fixing a generic choice of of $F$ and perturbation data $g\colon G\to\mathfrak{M}$, $\mathcal{M}_{T,\Gamma}(M,F)$ is a smooth manifold.
\end{lemma}
In the case of Section \ref{section-morse}, $F$ is not freely chosen but needs to satisfy the Kirchhoff laws at vertices. 
Recall the definition of $\mathcal{M}_T(\boldsymbol{q}_1,\dots,\boldsymbol{q}_m;\boldsymbol{q}_0)$.
We can similarly define
$\mathcal{M}_{T/\Gamma}(\boldsymbol{q}_1,\dots,\boldsymbol{q}_m;\boldsymbol{q}_0)$, $\mathcal{M}^{local}_{\Gamma}(\boldsymbol{q}_1,\dots,\boldsymbol{q}_m;\boldsymbol{q}_0)$, and $\mathcal{M}_{T,\Gamma}(\boldsymbol{q}_1,\dots,\boldsymbol{q}_m;\boldsymbol{q}_0)$ to specify the ribbon tree and the collapsed graph. 
We abbreviate them by $\mathcal{M}_{T}(\vec{\boldsymbol{q}})$, $\mathcal{M}_{T/\Gamma}(\vec{\boldsymbol{q}})$, $\mathcal{M}^{local}_{\Gamma}(\vec{\boldsymbol{q}})$, and $\mathcal{M}_{T,\Gamma}(\vec{\boldsymbol{q}})$, respectively.
The transversality arguments in \cite[Lemma 2.8]{Watanabe2018} do not always apply. 
However, when $M$ has dimension $n=2$, the compactness property we need is still satisfied even if $\mathcal{M}_{T,\Gamma}(\vec{\boldsymbol{q}})$ is not transversely cut out.
\begin{lemma}
    Suppose $M$ is of dimension $2$. Fixing a generic choice of perturbation data $g\colon G\to\mathfrak{M}$, $\mathcal{M}_{T/\Gamma}(\vec{\boldsymbol{q}})$ is a smooth manifold.
\end{lemma}
\begin{proof}
    This is similar to Lemma \ref{lemma-trans}.
\end{proof}

\begin{figure}[ht]
    \centering
    \includegraphics[width=10cm]{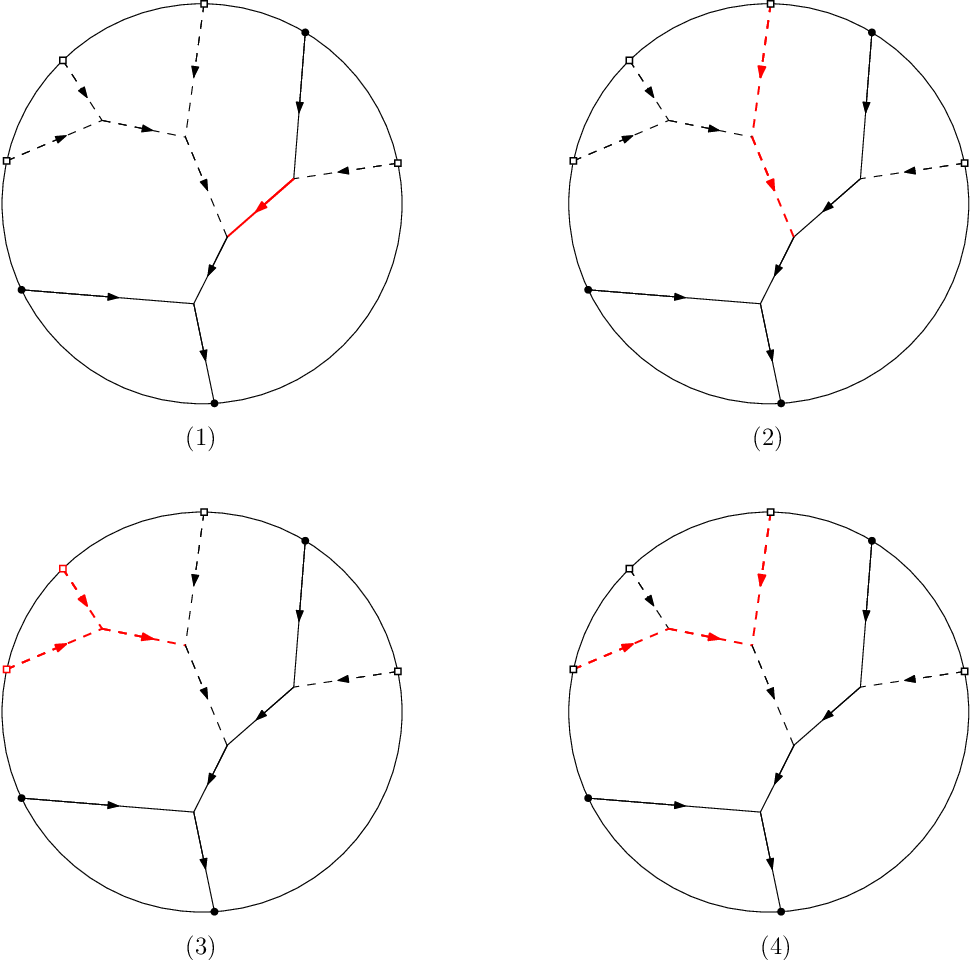}
    \caption{The red subgraphs denote $\Gamma$.}
    \label{fig-diagonal-cases}
\end{figure}

We now restrict to the case $n=2$.
The purpose to introduce $\mathcal{M}_{T,\Gamma}(\vec{\boldsymbol{q}})$ is that it works as part of the compactification of $\mathcal{M}_{T}(\vec{\boldsymbol{q}})$:
\begin{equation*}
    \mathcal{M}_{T,\Gamma}(\vec{\boldsymbol{q}})\subset \partial\overline{\mathcal{M}}_{T}(\vec{\boldsymbol{q}}).
\end{equation*}
Suppose $\operatorname{dim}\mathcal{M}_{T}(\vec{\boldsymbol{q}})=k\leq 1$.
We can make the following observations:
\begin{enumerate}
    \item If $E_\Gamma\cap (E_{mar}\cap E_{ext})=\emptyset$, then $\operatorname{dim}\mathcal{M}_{T/\Gamma}(\vec{\boldsymbol{q}})=k-|E_\Gamma|$ and $\mathcal{M}_{T,\Gamma}(\vec{\boldsymbol{q}})\simeq\mathcal{M}_{T/\Gamma}(\vec{\boldsymbol{q}})$.
    

    \item If $E_\Gamma\cap (E_{mar}\cap E_{ext})\neq\emptyset$ and the image of $\Gamma$ in $T/\Gamma$ as vertices has a total degree greater than 2, then $\operatorname{dim}\mathcal{M}_{T/\Gamma}(\vec{\boldsymbol{q}})\leq k-2<0$ and hence $\mathcal{M}_{T,\Gamma}(\vec{\boldsymbol{q}})=\emptyset$.

    \item If there exists $e_1,e_2\in E_\Gamma\cap E_{ext}\cap E_{mar}$ such that $e_1\neq e_2$ and $v_{en}(e_1)=v_{en}(e_2)$, then $\operatorname{dim}\mathcal{M}_{T/\Gamma}(\vec{\boldsymbol{q}})\leq k-2<0$ and hence $\mathcal{M}_{T,\Gamma}(\vec{\boldsymbol{q}})=\emptyset$.
    

    \item First exclude the above cases. For some $v\neq v_0\in V_{ext}$, denote the path from $v$ to $v_0$ by $P$. Let $E^{neigh}_P=E_{ext}\cap E_{mar}\cap\{v_{en}(e)\in V_P\}$ be the set of edges of $P$. If $\Gamma$ is a connected subtree of $T$ such that $E_\Gamma\subset E_P\cup E^{neigh}_P$, then $\operatorname{dim}\mathcal{M}_{T/\Gamma}(\vec{\boldsymbol{q}})\leq k-1\leq 0$.\label{case-nontrivial}
    For later use, we denote $\Gamma_{core}=\Gamma\cap P$.
\end{enumerate}
See Figure \ref{fig-diagonal-cases} for examples of each case.
Based on the above observations, we can conclude:
\begin{lemma}
    \label{lemma-diagonal-vanish}
    $\mathcal{M}_{T,\Gamma}(\vec{\boldsymbol{q}})$ is nonempty only if $|E_\Gamma|=1$ or $\Gamma$ satisfies case (\ref{case-nontrivial}).
\end{lemma}

The cases discussed above show how $\operatorname{dim}\mathcal{M}_{T/\Gamma}(\vec{\boldsymbol{q}})$ depends on the different topologies of $\Gamma$.
The existence of $\mathcal{M}_{T,\Gamma}(\vec{\boldsymbol{q}})$ also depends on the decorations on $\Gamma$.
Therefore, we further investigate case (\ref{case-nontrivial}). 
See Figure \ref{fig-diagonal-pair} for a typical local picture of $\Gamma$ and $G_\Gamma$ in this case.

\begin{figure}[ht]
    \centering
    \includegraphics[width=12cm]{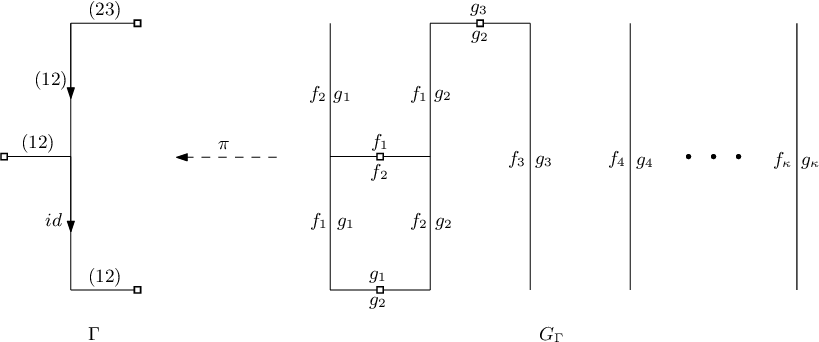}
    \caption{An example of $\Gamma$ and $G_\Gamma$. $\boldsymbol{g}$ and $\boldsymbol{f}$ correspond to the tuples of Morse functions on the left and right (with respect to the arrow) of $\Gamma_{core}$.}
    \label{fig-diagonal-pair}
\end{figure}

\begin{figure}[ht]
    \centering
    \includegraphics[width=14cm]{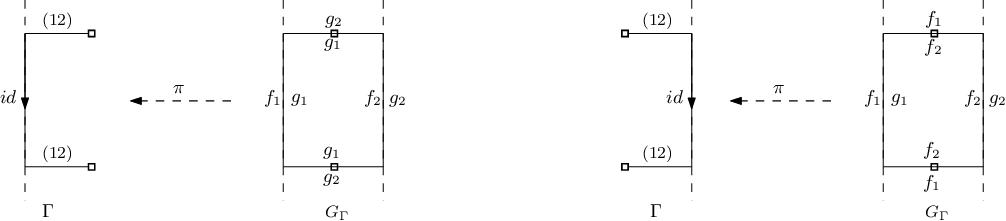}
    \caption{}
    \label{fig-diagonal-same}
\end{figure}
\begin{figure}[ht]
    \centering
    \includegraphics[width=14cm]{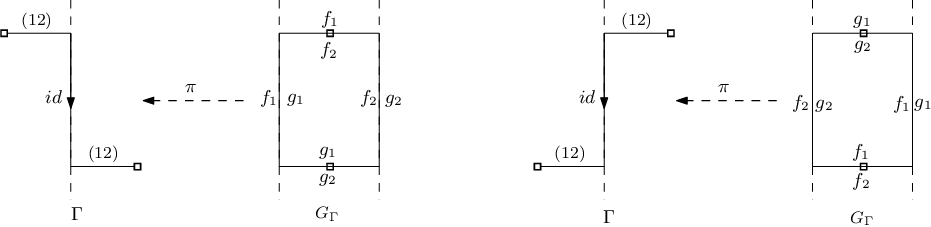}
    \caption{}
    \label{fig-diagonal-diff}
\end{figure}

\begin{lemma}
    \label{lemma-diagonal-vanish-2}
    In case (\ref{case-nontrivial}), $\mathcal{M}_{T,\Gamma}(\vec{\boldsymbol{q}})$ is empty.
\end{lemma}
\begin{proof}
    First, if there exists $i$ such that $G^i_\Gamma\cap\pi^{-1}(\Gamma_{core})$ has more than 2 components, it is easy to check that $\operatorname{dim}\mathcal{M}_{T/\Gamma}(\vec{\boldsymbol{q}})\leq k-2<0$. Consequently, $\mathcal{M}_{T/\Gamma}(\vec{\boldsymbol{q}})$ and $\mathcal{M}_{T,\Gamma}(\vec{\boldsymbol{q}})$ are both empty.

    It then suffices to consider the case of $\kappa=2$. 
    There are two subcases then.
    
    ~\\
    \noindent
    (1) Suppose there are two consecutive exterior edges of $\Gamma$ pointing to the same side of $\Gamma_{core}$. 
    In both sides of Figure \ref{fig-diagonal-same}, (\ref{eq-blow}) is equivalent to solving 
    \begin{equation*}
        \nabla(f_1-f_2)x_1-\nabla(g_1-g_2)(x_1+x_2+x_3)=0,
    \end{equation*}
    where $x_1,x_2,x_3\in\mathbb{R}_+\cup\{0\}$ and $|x_1|^2+|x_2|^2+|x_3|^2=1$.
    However, the solution exists only if $\nabla(f_1-f_2)$ is parallel to $\nabla(g_1-g_2)$, which is a constraint of codimension 1. 
    Since $\operatorname{dim}(\mathcal{M}_{T/\Gamma}(\vec{\boldsymbol{q}}))\leq0$, it follows that $\mathcal{M}_{T,\Gamma}(\vec{\boldsymbol{q}})$ is generically empty.
    
    ~\\
    \noindent
    (2) Suppose there are two consecutive exterior edges of $\Gamma$ pointing to different sides of $\Gamma_{core}$.  
    On both sides of Figure \ref{fig-diagonal-diff}, (\ref{eq-blow}) is equivalent to solving 
    \begin{equation*}
        \nabla(f_1-f_2)(x_1+x_3)-\nabla(g_1-g_2)(x_1+x_2)=0,
    \end{equation*}
    where $x_1,x_2,x_3\in\mathbb{R}_+\cup\{0\}$ and $|x_1|^2+|x_2|^2+|x_3|^2=1$.
    Similarly, the solution exists only if $\nabla(f_1-f_2)$ is parallel to $\nabla(g_1-g_2)$, which is a constraint of codimension 1. 
    Since $\operatorname{dim}(\mathcal{M}_{T/\Gamma}(\vec{\boldsymbol{q}}))\leq0$, it follows that $\mathcal{M}_{T,\Gamma}(\vec{\boldsymbol{q}})$ is generically empty.
\end{proof}

\begin{lemma}
    \label{lemma-no-diagonal}
    Fixing a generic choice of perturbation data, if $|E_\Gamma|>1$, then $\mathcal{M}_{T,\Gamma}(\vec{\boldsymbol{q}})$ is empty.
    If $|E_\Gamma|=1$, $\mathcal{M}_{T,\Gamma}(\vec{\boldsymbol{q}})$ is a smooth manifold such that 
    \begin{equation*}
        \operatorname{dim}(\mathcal{M}_{T,\Gamma}(\vec{\boldsymbol{q}}))=\operatorname{dim}(\mathcal{M}_{T}(\vec{\boldsymbol{q}}))-1.
    \end{equation*}
\end{lemma}
\begin{proof}
    The first part follows from Lemma \ref{lemma-diagonal-vanish} and Lemma \ref{lemma-diagonal-vanish-2}.
    The second part follows from the transversality of $\mathcal{M}_{T/\Gamma}(\vec{\boldsymbol{q}})$ and that $\mathcal{M}_{T,\Gamma}(\vec{\boldsymbol{q}})\simeq\mathcal{M}_{T/\Gamma}(\vec{\boldsymbol{q}})$ when $|E_\Gamma|=1$.
\end{proof}

\printbibliography

\end{document}